\documentclass[12pt]{amsart} 
\usepackage[utf8]{inputenc} 
\usepackage{amsthm}
\usepackage{amssymb}
\usepackage{fancyref}
\usepackage{array} 
\usepackage{ dsfont }
\usepackage{color, soul}
\usepackage{ bbold }
\usepackage{multirow}

\usepackage[utf8]{inputenc}
\usepackage[english]{babel}
 
\newtheorem{definition}{Definition}[section]
\newtheorem{theorem}{Theorem}[section]

\newtheorem{lemma}{Lemma}[section]
\newtheorem{prop}{Proposition}[section]
\newtheorem{rem}{Remark}[section]

\usepackage[left=3cm,right=3cm,top=3cm,bottom=3cm]{geometry}
\usepackage{fancyhdr} 
\newcommand{\R}{\mathds{R}}
\newcommand{\E}{\mathds{E}}
\newcommand{\PP}{\mathds{P}}
\newcommand{\dint}{\mathrm{d}}

\usepackage{mathtools}
\allowdisplaybreaks

\begin{document}

\title{The Stochastic Geometry of Unconstrained One-bit Data Compression}
\author{Fran\c{c}ois Baccelli and Eliza O'Reilly}

\address{University of Texas at Austin, Department of Mathematics, RLM 8.100,
2515 Speedway Stop C1200
Austin, Texas 78712-1202}
\curraddr{}
\email{baccelli,eoreilly@math.utexas.edu}
\thanks{The first and second author were supported by a grant of the Simons Foundation (\#197982 to UT Austin) and the second author was supported by the National Science Foundation Graduate Research Fellowship under Grant No. DGE-1110007.}


\begin{abstract}  
A stationary stochastic geometric model is proposed for analyzing the data compression method used
in one-bit compressed sensing. The data set is an unconstrained stationary set, for instance
all of $\R^n$ or a stationary Poisson point process in $\R^n$. It is compressed using a stationary
and isotropic Poisson hyperplane tessellation, assumed independent of the data.
That is, each data point is compressed using one bit with respect to each hyperplane, which is the side
of the hyperplane it lies on. This model allows one to determine how the intensity of the hyperplanes
must scale with the dimension $n$ to ensure sufficient separation of different data by the hyperplanes
as well as sufficient proximity of the data compressed together.
The results have direct implications in compressive sensing and in source coding.
\end{abstract}
\maketitle

\section{Introduction and Motivations}


One-bit compressed sensing is a method of signal recovery from a sequence of measurements contained in $\{-1, 1\}$. More specifically, one aims to recover the signal $x \in \R^n$ from measurements of the form
\begin{align*}
y_i = \mathrm{sign}( \langle u_i, x \rangle - t_i), 
\end{align*} 
where the $u_i$ are independent vectors in $\mathds{R}^{n}$ and $t_i$ random displacements in $\R$. One can interpret this problem geometrically, by the fact that each pair $(u_i, t_i)$ defines a unique affine hyperplane in $\R^n$ with normal vector $u_i$ at distance $t_i$ from the origin. The measurement $y_i \in \{-1, 1\}$ then indicates which side of the hyperplane the signal $x$ lies on. This collection of hyperplanes tessellates the space of signals into convex cells. Two signals contained in the same cell will have the same set of one-bit measurements $\{y_i\}$. The quality of this compression can be measured in a few different ways. For instance, one can measure how likely it is that two different signals are compressed differently, i.e., lie in different cells of the tessellation. As in one-bit compressed sensing, the quality can also be determined by having a small error in signal recovery, which can be guaranteed if the collection of hyperplanes tessellate the signal space into cells small enough to ensure all signals within a single cell are close in Euclidean distance.

Previous work (\cite{Bilyk}, \cite{Ward}, \cite{PV2}) has examined this problem when it is known that the signal lies in some bounded set $K \subset \R^n$. In this paper, we consider the data set to be either all of $\R^n$ or an uncountable discrete subset of $\R^n$ modeled with a stationary Poisson point process. The assumption that the data is Poisson provides a worse-case scenario, since any dependence between the underlying points increases one's ability to compress the data in such a way that the signals can be recovered with small error. The set of random hyperplanes used to obtain the one-bit measurements is given by a stationary and isotropic Poisson hyperplane process. The reasons for this choice are discussed at the end of the paper (see Subsection \ref{sec:IPH}), the key reason being that it leads to the least volume of data compressed with a typical data point among a wide collection of hyperplane models.

As already explained, the aim is to find the minimum intensity of the hyperplane process at some scaling with the space dimension $n$ such that different data will be separated by hyperplanes with high probability, and also for data compressed in the same way to be close with high probability. Under the assumption of stationarity, we can ask for, in some sense, a ``typical" instance to satisfy the desired property. To address the ``typicality", there are two viewpoints to take. One is from the view of a typical data point, and in the stationary regime, we can consider its location to be at the origin. The cell of the tessellation that the typical signal is contained in is then the so-called zero cell \cite{Stoyan}, also referred to as the Crofton cell. The other viewpoint is to ask that a typical cell satisfy some property, e.g., to have small diameter.
The typical cell of a stationary Poisson hyperplane tessellation can be interpreted as the distribution of the cell obtained when taking a large ball
centered at the origin, and picking a cell intersecting that ball uniformly at random.
The zero cell is larger in mean than the typical cell, as there is bias towards larger cells
when asking that it contain the origin. The viewpoint of a typical signal and its cell,
the zero cell, seems a more natural viewpoint to take here, and will be the main focus of this paper, although some results are also derived on the typical cell for comparison.

To summarize the results, consider a sequence of compressions indexed by dimension, i.e., for each $n$, let $X_n$ be a stationary and isotropic Poisson hyperplane tessellation in $\R^n$ with intensity $\gamma_n$ that is used to compress the underlying data. We let $\gamma_n \sim \rho n^{\alpha}$ as $n \to \infty$ and discuss the values of $\alpha$ for which a good separation or low distortion of the data can be achieved with high probability by the hyperplanes when $n$ is large. Several criteria of good separation and low distortion are discussed.
By good separation, we mean a property that connects differences between data and differences between their encodings. By low distortion, we mean a property than connects closeness of data and similarity of their encodings. The results on the matter are summarized below when data are the whole of ${\R}^n$.

The first separation criterion discussed is that the distance to the nearest
data that is compressed differently from the typical data (i.e.,
the closest point of the Euclidean space which is not in the zero cell) be small.
It is shown that as long as $\alpha > 0$, this distance tends to zero in distribution as $n$ tends to infinity. 

The second separation criterion considered is that some transformation of the typical signal
is compressed differently than the typical signal with high probability. We discuss two types of transformations:
(i) a Gaussian displacement with fixed variance $\sigma$ per dimension (which is the least demanding of the criteria discussed here),
and (ii) a displacement at a fixed distance $\sigma$ away and in a random direction.
For case (i), we show that, for $\alpha = 0$, the typical signal is compressed 
in the same way as the typical signal with a probability decreasing exponentially with $\rho$.
We also show that the same holds in case (ii) provided $\alpha = \frac{1}{2}$. 


The first low distortion criterion is the requirement that the volume of other data compressed with a typical data be small. The hyperplane intensities discussed above are not large enough for this to hold. While data in most directions will be separated from the typical data, 
there is a set of directions of decreasing measure in which the compression will remain identical, and in high dimension, this is where most of the volume of data compressed like the typical signal lies.
Considering this low distortion criterion, 
we show that, for $\alpha = 1$, there is a threshold for $\rho$ above which the
expected value of the volume in question goes to zero and below which it approaches infinity.

A small volume still does not ensure that all data compressed together is close in Euclidean distance. This motivates the discussion of a second low distortion criterion. In the case where data is the whole Euclidean space, the requirement is that the point which is the farthest away from the typical data and encoded in the same way be within some distance $R$.
It is shown that if we increase $\alpha$ to $\frac{3}{2}$, then there exists a value for $\rho$ above which this probability approaches one as dimension $n$ tends to infinity. A similar criterion for the case when the data is modeled with a Poisson point process is also discussed.

Some of these scalings can be significantly decreased if it is known that the data are 'sparse', namely
lie within a lower dimensional subspace of $\R^n$. In Section \ref{s:reduction},
we show how this affects the intensity of hyperplanes needed for the above low distortion criteria.

The results have several implications in compressed sensing and in source coding.
These are discussed in Subsections \ref{s:commentCS} and \ref{s:commentIT} at the end of the paper.

\section{Preliminaries and Notation}

First we define the notation for the classical objects used in the present paper.
Let $B_n(r)$ denote the ball or radius $r$ centered at the origin in $\R^n$. The usual $\ell^2$ norm of a vector is denoted by $|\cdot|$, and the $n$-dimensional volume of a set $K \subset \R^n$ by $V_n(K)$.
The volume of the $n-$dimensional unit ball $B_n(1)$ is denoted by $\kappa_n$ and the surface area of the $n$-dimensional unit sphere $S^{n-1}$ is denoted by $\omega_n$. They satisfy
\[\kappa_n = \frac{\pi^{\frac{n}{2}}}{\Gamma( \frac{n}{2} + 1)}, \qquad \omega_n = n \kappa_n = \frac{2\pi^{\frac{n}{2}}}{\Gamma(\frac{n}{2})}.\]
Also recall the following special functions. The gamma function is defined as
\[\Gamma(x) := \int_0^{\infty} t^{x-1}e^{-t} dt,\]
and the upper and lower regularized incomplete gamma functions are defined for all $R \geq 0$ by
\[\Gamma_u(x, R) := \frac{\int_R^{\infty} t^{x-1}e^{-t} dt}{\Gamma(x)}, \qquad \Gamma_{\ell}(x, R) := \frac{\int_0^{R} t^{x-1}e^{-t} dt}{\Gamma(x)},\]
respectively. Stirling's formula 
gives the following asymptotic expansion as $x \rightarrow \infty$:
\begin{align}\label{e:gamma}
\Gamma(x + 1) &\sim \sqrt{2\pi x}\left(\frac{x}{e}\right)^x. 
\end{align}
The following asymptotic formulas will be used throughout: by \eqref{e:gamma}, as $n \rightarrow \infty$,
\begin{align}\label{e:asymptotics}
\kappa_n \sim \frac{1}{\sqrt{n\pi}}\left(\frac{2\pi e}{n}\right)^{n/2}\qquad \text{and} \qquad \frac{\kappa_{n-1}}{n\kappa_n} &\sim \frac{1}{\sqrt{2\pi n}}.
\end{align}

Denote by $\mathcal{F}, \mathcal{C}$ the sets of closed and convex subsets of $\R^n$, respectively. For $A \subset \R^n$, define
\begin{align}
\mathcal{F}^A := \{ F \in \mathcal{F} : F \cap A = \emptyset\}
\text{ and }
\mathcal{F}_A := \{F \in \mathcal{F} : F \cap A \neq \emptyset\}.
\end{align}
The $\sigma$-algebra $\mathcal{B}(\mathcal{F})$ of Borel sets of $\mathcal{F}$ is generated by either of the systems $\{\mathcal{F}_C : C \in \mathcal{C}\}$ and $\{\mathcal{F}^C : C \in \mathcal{C}\}$ (see Lemma 2.1.1 in \cite{weil}). Denote the set of $n-1$ dimensional hyperplanes in $\R^n$ by $\mathcal{H}^n$ and the Grassmanian of $n-1$-dimensional linear subspaces of $\R^n$ by $G(n, n-1)$. The set $G(n, n-1)$ is the subset of hyperplanes in $\mathcal{H}^n$ that pass through the origin. 


\subsection{Poisson Hyperplane Tessellations}


A hyperplane process $X$ in $\R^n$ is a random counting measure on the space $\mathcal{H}^n$.
The process $X$ is stationary if its distribution is invariant under translations and it is
isotropic if its distribution is invariant under rotations about the origin. 

The intensity measure of $X$ is defined as $\Theta(\cdot) := \E[X(\cdot)]$.
The following theorem (see, e.g., \cite{weil}) provides a decomposition for the intensity measure
for all stationary hyperplane processes. Note that elements of the space $\mathcal{H}^n$
are of the form  
\begin{equation}
H(u,\tau) := \{x \in \R^n : \langle x, u \rangle = \tau\},
\end{equation}
where $u \in \R^n$ and $\tau \in \R$.

\begin{theorem}
Let $X$ be a stationary hyperplane process in $\R^n$ with intensity measure $\Theta \neq 0$. Then, there is a unique number $\gamma \in (0, \infty)$ and probability measure $\mathds{Q}$ on $G(n, n-1)$ such that for all nonnegative measurable functions $f$ on $\mathcal{H}^n$,
\[\int_{\mathcal{H}^n} f d\Theta = 2 \gamma \int_{S^{n-1}} \int_{0}^{\infty} f(H(u, \tau))d\tau \phi(du),\] 
where for $A \in \mathcal{B}(S^{n-1})$, $\phi(A) := \frac{1}{2} \mathds{Q}(\{u^{\perp} : u \in A\})$. $\phi$ is called the spherical directional distribution. In particular, for $A \in \mathcal{B}(\mathcal{H}^n)$,
\[\Theta(A) = 2\gamma \int_{S^{n-1}} \int_{0}^{\infty} 1_{\{H(u,\tau) \in A\}} d\tau \phi(du).\]
\end{theorem}

The parameter $\gamma$ is called the intensity and $\mathds{Q}$ the directional distribution of $X$. If $X$ is isotropic, then $\mathds{Q}$ is rotationally invariant and thus is the Haar measure $\nu_{n-1}$ and $\phi = \sigma$, the normalized spherical Lebesgue measure on $S^{n-1}$. 

The hyperplane process $X$ with intensity measure $\Theta$ is Poisson if for all disjoint $A_1, ..., A_k \in \mathcal{B}(\mathcal{H}^n)$ such that $\Theta(A_i) < \infty$ for all $i$,
\begin{align*} 
\mathds{P}(X(A_1) = m_1, ..., X(A_k) = m_k) = \prod_{i=1}^k\frac{\Theta(A)^{m_i}}{m_i!} e^{- \Theta(A)}.
\end{align*}


\subsection{Zero cell}\label{s:zero}


A hyperplane process $X$ in $\R^n$ induces a random tessellation of $\R^n$.
The zero cell, or Crofton cell, of this tessellation, denoted $Z_0$, 
is the cell of this tessellation containing the origin.

The following result (see Theorem 10.4.9 in \cite{weil}) states that for stationary Poisson hyperplane processes, isotropic hyperplanes minimize the expected area of the zero cell over all spherical distributions. This result helps to justify considering the class of isotropic Poisson hyperplanes to tessellate the space, since cells of smaller volume may lead to a more efficient compression.

\begin{theorem}
Let $X$ be a nondegenerate stationary Poisson hyperplane process in $\R^n$ of intensity $\gamma$, and let $Z_0$ be the zero cell of the induced hyperplane tessellation. Then,
\begin{align*}
\mathds{E}V_n(Z_0) \geq n!\kappa_n \left( \frac{n \kappa_n}{2\gamma\kappa_{n-1}} \right)^{n},
\end{align*}
with equality if and only if $X$ is isotropic. 
\end{theorem}

As mentioned in the introduction, a small volume is not sufficient
to ensure that two data points that have the same compression are close together.
This requires the cell the points are contained in to have small diameter, but this is a difficult quantity to study.
A related quantity is the radius of the smallest ball centered at the origin that contains the cell $\mathcal{C}$, i.e., the quantity
\begin{align*}
R_M(\mathcal{C}) = \inf\{r > 0 : \mathcal{C} \subset B(r) \}.
\end{align*}
The distribution of $R_M(Z_0)$ is described in \cite{calka}.
It is based on the observation that if $R_M \geq r$, then the sphere of radius $r$
centered at the origin will not be covered by the random arcs generated by the hyperplanes
that compose the faces of $Z_0$, i.e., $rS^{n-1} \cap \mathrm{int}(Z_0) \neq \emptyset$.
Since the directional distribution of $X$ is just the Haar measure on $S^{n-1}$,
the probability that $R_M \geq r$ is the probability that
$S^{n-1}$ can be covered by a Poisson number $N$ of independent spherical caps,
with angular radii divided by $\pi$ distributed as
$d\nu(\theta) = \pi \sin(\pi \theta) 1_{[0, 1/2]}(\theta) d\theta$.
Unfortunately, no explicit formula for this probability is known beyond dimension two.


\subsection{Typical cell}


Since larger cells are more likely to contain the origin, the zero cell is not a good measure of the average or ``typical" cell. We can instead consider a large compact set and pick a cell uniformly at random and translate it is some appropriate way so that it contains the origin. This more accurately represents the average distribution of the cells induced by the hyperplane process. Formally, we define the typical cell as follows. Let $c: \mathcal{C}^{'} \rightarrow \R^n$ be a {\em center function}, that is, a measurable map which is compatible with translations, i.e., $c(C + x) = c(C) + x$ for all $x \in \R^n$. For a hyperplane process $X$, let $\hat{X}$ denote the induced random mosaic, that is, the collection of cells of the induced tessellation.
\begin{definition}
The typical cell $Z$ of a hyperplane process $X$ 
is the random polytope with distribution
\begin{align*}
\mathds{Q}_0(\mathcal{A}) = \frac{1}{\lambda|B|} \mathds{E} \sum_{P \in \hat{X}} 1_{\mathcal{A}}\{P - c(P)\}1_{B}(c(P)),
\end{align*}
where $B \in \mathcal{B}(\R^n)$ is an arbitrary bounded Borel set, and $\lambda$ is the cell intensity of $\hat{X}$. Also, this distribution has the ergodic interpretation
\begin{align*}
\mathds{Q}_0(\mathcal{A})  = \lim_{r \rightarrow \infty} \frac{1_{\mathcal{A}}\{P - c(P)\}1_{r [-1/2, 1/2]^n}(c(P))}{\sum_{P \in \hat{X}} 1_{r[-1/2, 1/2]}(c(P))}, \qquad a.s.
\end{align*}
\end{definition}

The cell intensity $\lambda$ of the induced random mosaic $\hat{X}$ of a hyperplane process $X$ in $\R^n$ is related to the intensity $\gamma$ of $X$ in the following way:
\begin{equation}\label{e: cell_int}
 \lambda = \kappa_n\left( \frac{ \gamma\kappa_{n-1}}{n \kappa_n }\right)^n.\end{equation}

Let $Z$ denote the typical cell of $X$. 
It is known that (see, e.g., \cite[(10.4) and (10.46)]{weil}),
\begin{align}\label{e:typ_vol}
\mathds{E}[V(Z)] = \int V(K) \mathds{Q}(K) = \frac{1}{\lambda} = \frac{1}{\kappa_n}\left( \frac{n \kappa_n }{ \gamma\kappa_{n-1}}\right)^n. 
\end{align}

\begin{rem}\label{r:shannon}
Consider a sequence of hyperplane tessellations $X_n$ in increasing dimensions $\R^n$ with intensity $\gamma_n$ and cell intensity $\lambda_n$. If $\lambda_n \sim e^{n\lambda}$ as $n \to \infty$, this corresponds to  when $\gamma_n \sim \rho n$ as $n \to \infty$. This exponential scaling with dimension for the point process of cell centroids matches the so-called Shannon regime studied in \cite{venkat}, and leads to a linear scaling of the hyperplane intensity with dimension.
\end{rem}

The inradius $r_{in}$ of a cell is the radius of the largest ball completely contained in the cell. The following result gives the distribution of the inradius of the typical cell.  

\begin{theorem}\label{t:typ_inrad} (Theorem 10.4.8 in \cite{weil}) Let $X$ be a nondegenerate stationary Poisson hyperplane process in $\R^n$ with intensity $\gamma$. Let $Z$ be the typical cell. Then,
\begin{align*}
\mathds{P}(r_{in}(Z) \leq a) = 1 - e^{-2\gamma a}, \qquad a \geq 0.
\end{align*}
\end{theorem}


\subsection{Palm Distribution}


Throughout this paper, when the underlying data is assumed to be discrete,
it is modeled by a stationary Poisson point process $N$ with intensity $\lambda$.
Since this is an unbounded collection of data, we need some way of examining a typical
data point and the cell of the tessellation that contains it. 

To do this, we use the Palm probability measure of $N$, denoted by $\mathds{P}^0_N$, which is defined as follows.
Let $(\Omega, \mathcal{A}, \{\theta_t\}_{t \in \R^n}, \mathds{P})$ be a stationary framework and $N$ a random measure compatible with the flow $\{\theta_t\}_{t \in \R^n}$, implying $N$ is stationary.
The Palm probability associated with $N$, denoted $\mathds{P}^0_N$, is defined on $(\Omega, \mathcal{A})$ by
\[ \mathds{P}_N^0(A) := \frac{1}{\lambda} \mathds{E}\left[\int_B 1_A \circ \theta_x N(dx) \right],\]
for any bounded Borel set $B$ with volume one. The Palm probability $\mathds{P}_N^0$ can be thought of the distribution of
$N$ conditioned on there being a point at $0$. Thus, to talk about the cell of a typical
data point, we condition on a point being at $0$, and examine the cell of the tessellation it is contained 
in, i.e., the zero cell. There is also the following ergodic interpretation of the Palm probability.
By Birkhoff's Pointwise Ergodic theorem, for all convex averaging sequences $\{K_m\}_{m \geq 1}$ in $\R^n$,
and all $f: \Omega \rightarrow \mathds{R}_{+}$ measurable and in $L_1(\mathds{P}_N^0)$,
\[ \frac{1}{V_n(K_m)} \int_{K_m} f \circ \theta_x N(dx) \rightarrow \lambda \mathds{E}_N^0[f], \text{ as } m \rightarrow \infty, \qquad \mathds{P}-a.s.\]
Thus, we can think of the Palm probability as the empirical average over all the points in a very large ball. The reduced Palm probability measure of $N$, denoted $\mathds{P}^{0,!}_N$ is defined as $\mathds{P}^0_{N - \delta_0}$, that is, the Palm measure with the point at 0 removed. An important result called Slivnyak's theorem states that a Poisson point process has the same distribution as its reduced Palm distribution, i.e. $\mathds{P}^{0,!}_N = \mathds{P}_N$. 

The distribution of the typical cell of a stationary tessellation can also be thought of as the zero cell of the tessellation under the Palm measure of the point process of cell centers. That is, its distribution is that of the cell containing the origin, conditioned on a cell of the tessellation having its center at the origin.


\section{Results}


In this section, for each $n$, let $X_n$ be a stationary and isotropic Poisson hyperplane process
in $\R^n$ with intensity $\gamma_n$ representing the compression scheme (note that the Poisson
assumption implies that the compression scheme is characterized by a single parameter $\gamma_n>0$, for all dimensions $n$).
The zero cell of the tessellation is denoted $Z_{0,n}$ and the typical cell is denoted $Z_n$.
In the case where the underlying data is discrete, $N_n$ is a stationary Poisson point process
with intensity $\lambda_n$ lying in $\R^n$ and independent of $X_n$, representing the data.
The Palm probability of $N_n$ is denoted by $\mathds{P}_{n}^0$.

As explained in the introduction, the goal is to find the minimum intensity
$\gamma_n$ needed to separate or minimize the distortion of the data $\R^n$ or $N_n$ with high probability
according to various criteria listed there.


\subsection{Distance from typical data to nearest data compressed differently}


Given a typical data point, we first ask how far away the closest data is that is compressed differently in any direction. When the data is all of $\R^n$, this is the distance to the nearest separating hyperplane in any direction.
To find the distribution of this distance, notice that if no hyperplane hits the ball of radius $r$ centered
on the typical data, then this distance is greater than $r$. 
This is the spherical contact distribution \cite{weil}:
\[ D_n(r) :=  \mathds{P}\left( X_n\left(\mathcal{F}_{B_n(r)}\right) = 0\right). \]

\begin{prop}
\label{prop:3.1}
Assume $\gamma_n \rightarrow \infty$ as $n \rightarrow \infty$, for example $\gamma_n \sim \rho n^{\alpha}$ as $n \to \infty$ for any $\alpha > 0$. Then, for fixed $r > 0$,
\[\lim_{n \rightarrow \infty}  D_n(r) = 0.\]
\end{prop}

\begin{proof}
By the fact that $X$ is Poisson,
\begin{align*}
\lim_{n \rightarrow \infty}  D_n(r) = \lim_{n \rightarrow \infty} \mathds{P}( X_n(\mathcal{F}_{B_n(r)}) = 0) =  \lim_{n \rightarrow \infty} e^{-\Theta_n(\mathcal{F}_{B_n(r)})} =  \lim_{n \rightarrow \infty} e^{-2\gamma_n r} = 0.
\end{align*}
\end{proof}

Another viewpoint to take is the distance to the nearest data compressed differently from the center of a typical cell of the tessellation, where the center is considered to be the center of the largest ball completely contained in the cell. This is equivalent to asking for the distribution of the inradius of the typical cell. Theorem \ref{t:typ_inrad} implies the following.

\begin{prop}
\label{prop:3.2}
Assume $\gamma_n \rightarrow \infty$ as $n \rightarrow \infty$, for example $\gamma_n \sim \rho n^{\alpha}$ for any $\alpha > 0$. Then, for fixed $r > 0$,
\[\lim_{n \rightarrow \infty}  \mathds{P}(r_{in}(Z_n) > r) = 0.\]
\end{prop}

\subsection{Separation of two different data}


The next criterion for separation is the probability that two different data points, one obtained by
some given transformation of the other, are compressed differently, i.e., 
the probability that there is at least one hyperplane separating them.

First, consider the case where the transformation is a random displacement by an i.i.d.
Gaussian with mean zero and variance $\sigma^2$ per dimension. 

\begin{prop}
\label{pro:Gauss}
For each $n$, let $Y_n \sim \mathcal{N}(0, \sigma^2 I_n)$ be a Gaussian random vector 
in $\R^n$. Assume 
$\gamma_n \sim \rho n^{\alpha}$ for some $\rho > 0$ as $n \to \infty$. 
Then,
\begin{align*}
\lim_{n \rightarrow \infty} \mathds{P}(Y_n \in Z_{0,n})=  \begin{cases} 0, & \alpha > 0 \\ e^{-\sqrt{\frac{2}{\pi}}\rho \sigma}, & \alpha = 0 \\ 1, & \alpha  < 0 .\end{cases}
\end{align*}
\end{prop}

\begin{proof}
First, by the decomposition of the spherical Lebesgue measure (Equation (1.41) in \cite{Muller}), for all $x\in \mathds{R}^n$,
\begin{align}\label{numhypsep}
\Theta(\mathcal{F}_{[0,x]}) &=  2\gamma_n \int_{S^{n-1}} \int_{0}^{\infty} 1_{\{H(u,t) \cap [0,x] \neq 0\}} dt \sigma(du) = 2\gamma_n \int_{S^{n-1}} \int_{0}^{\infty} 1_{\{0 \leq t \leq \langle x, u \rangle_+\}}  dt \sigma(du) \nonumber  \\
&= 2\gamma_n \int_{S^{n-1}} \langle x, u \rangle_+  \sigma(du) = 2\gamma_n |x|  \int_{S^{n-1}} \left< \frac{x}{|x|}, u \right>_+  \sigma(du) = 2 \gamma_n \frac{\kappa_{n-1}}{n\kappa_n}|x|,
\end{align}
where $a_+=\max(a,0)$.

Then, since $X$ is Poisson, by \eqref{numhypsep},
\begin{align}\label{e:x_in_Z0}
\mathds{P}(x \in Z_{0,n}) = \mathds{P}\left(X\left(\mathcal{F}_{[0,x]} \right) = 0 \right) = e^{- \Theta(\mathcal{F}_{[0,x]})} = e^{- \frac{2\gamma_n \kappa_{n-1}}{n\kappa_n}|x| }.
\end{align}
By \eqref{e:x_in_Z0},
\begin{align*}
\mathds{P}(Y_n \in Z_{0,n}) = \mathds{E}\left[ \mathds{P}(Y_n \in Z_{0,n} | Y_n)\right] = 
 \mathds{E}\left[e^{- \frac{2\gamma_n \kappa_{n-1}}{n\kappa_n}|Y_n| }\right] .
\end{align*}
By the strong law of large numbers, $|Y_n|^2/n \rightarrow \sigma^2$ a.s.,
and by \eqref{e:asymptotics}, as $n \to \infty$,
\begin{align}\label{e:asymp2}
\frac{2\gamma \kappa_{n-1}}{n\kappa_n} \sim \frac{2\rho n^{\alpha} \kappa_{n-1}}{n\kappa_n} \sim \frac{2\rho n^{\alpha}}{\sqrt{2\pi n}} = \sqrt{\frac{2}{\pi}} \rho n^{\alpha - \frac{1}{2}}. \end{align}
Then, as $n \rightarrow \infty$,
\begin{align*}
\frac{2\gamma_n \kappa_{n-1}}{n\kappa_n}|Y_n| \sim \sqrt{\frac{2}{\pi}} \rho n^{\alpha}\frac{|Y_n|}{\sqrt{n}} \to \sqrt{\frac{2}{\pi}}\rho n^{\alpha - \frac{1}{2}}, \text{    a.s.}
\end{align*}
Thus, 
\begin{align*}
\lim_{n \rightarrow \infty} \mathds{E}\left[e^{- \frac{2\gamma_n \kappa_{n-1}}{n\kappa_n}|Y_n| }\right] = \begin{cases} 0, & \alpha > 0 \\ e^{-\sqrt{\frac{2}{\pi}}\rho \sigma}, & \alpha = 0 \\ 1, & \alpha  < 0 .\end{cases}
\end{align*}

\end{proof}

Next, consider the case where the displacement is uniformly chosen
on the sphere of fixed radius $\delta$. By the fact that the tessellation is isotropic, 
this is equivalent to looking at the linear contact distribution for any fixed direction
$u \in S^{n-1}$ at distance $\delta$:
\[ L_u(\delta) := \mathds{P}\left(X \left(\mathcal{F}_{[0, \delta u]}\right) = 0\right).\]

\begin{prop}\label{p:sep_unif}
For each $n$, let $Y_{n, \delta}$ be a uniformly chosen random point on the sphere of radius $\delta$ in $\R^n$. Under the same assumptions as in Proposition \ref{pro:Gauss},
\begin{align*}
\lim_{n \rightarrow \infty} \mathds{P}( Y_{n,\delta} \in Z_{0,n}) =  \begin{cases} 0, & \alpha > \frac{1}{2} \\ e^{-\sqrt{\frac{2}{\pi}}\rho \delta}, & \alpha = \frac{1}{2} \\ 1, & \alpha  < \frac{1}{2} .\end{cases}
\end{align*}
\end{prop}

\begin{proof}
By \eqref{e:x_in_Z0}, 
\begin{align*}
 \mathds{P}( Y_{n, \delta} \in Z_{0,n}) = \mathds{E}\left[e^{- \frac{2\gamma_n \kappa_{n-1}}{n\kappa_n}|Y_{n,\delta}| }\right] = e^{- \frac{2\gamma_n \kappa_{n-1}}{n\kappa_n} \delta } .
\end{align*}
Then, by the asymptotic formula \eqref{e:asymp2}, as $n \to \infty$,
\[\frac{2\gamma \kappa_{n-1}}{n\kappa_n}\delta \, \sim \sqrt{\frac{2}{\pi}} \rho n^{\alpha - \frac{1}{2}} \delta. \]
By continuity, the conclusion holds.
\end{proof}

Note that a scaling of $\gamma_n$ greater than $n^{\frac{1}{2}}$ (resp. more than a constant) 
is needed for this last separation criterion (resp. that of the Gaussian displacement) to hold
as dimension increases. This is less than what is needed for the expected volume of $V_n(Z_{0,n})$ to be small as seen in the next section. This indicates that in high dimensions, most of the volume of the cell is concentrated in a set of directions with very small measure. 


\subsection{Volume of data compressed together}


This section is focused on the asymptotic behavior as $n$ goes to infinity
of the volume of the data that is compressed together in a cell of the tessellation. 
The requirement that this volume tends to zero is a first low distortion criterion.
One viewpoint is to examine the volume of data in the cell containing a typical data point.
When the data is all of $\R^n$, 
this is the just the volume of $Z_{0, n}$. This quantity has been studied
in \cite{HugZero} and \cite{HorrmanZero}. The expected value is
\begin{align}\label{e:vol}
\mathds{E}[V_n(Z_{0,n})] &= n!\kappa_n \left(\frac{n \kappa_n }{2\gamma \kappa_{n-1}} \right)^n = \left((n!\kappa_n)^{1/n} \frac{n \kappa_n }{2\gamma \kappa_{n-1}} \right)^n.
\end{align}
From \cite{HorrmanZero}, the following bounds on higher moments of $V_n(Z_{0,n})$ are obtained:
\begin{align}\label{e:Vmombnd}
 \Gamma(n + 1) \kappa_n^k \left(\frac{n\kappa_n}{2\gamma \kappa_{n-1}} \right)^{kn} \leq \mathds{E}[V_n(Z_{0,n})^k] \leq  \Gamma(kn + 1) \kappa_n^k \left(\frac{n\kappa_n}{2\gamma \kappa_{n-1}} \right)^{kn}.
\end{align}
A corollary in \cite{HorrmanZero} shows there exist constants $c$ and $C$, not depending on $n$ or $\gamma$, such that
\begin{align}\label{zero_var}
c\sqrt{n}\left(\frac{\pi}{e}\frac{n}{\gamma}\left(1 + \frac{1}{n}\right)^{\frac{n}{2}}\right)^{2n} \leq \mathrm{Var}[V_n(Z_{0,n})] \leq C\sqrt{n}\left(\frac{\pi}{e}\frac{n}{\gamma}\left(1 + \frac{1}{n}\right)^{\frac{n}{2}}\right)^{2n}.
\end{align}
The authors note that if $\gamma$ scales with $n$
in such a way that $\mathds{E}[V_n(Z_{0,n})] = 1$ for all $n$,
the lower bound implies that the variance of $V_n(Z_{0,n})$ approaches infinity as the dimension $n$ increases, which contrasts with the behavior seen in the typical cell of the Poisson-Voronoi tessellation, where the variance converges to zero, see \cite{Voronoi}. 

By the asymptotic formulas \eqref{e:asymptotics} and the above results,
we obtain the following limiting behavior as dimension goes to infinity.
\begin{prop}\label{p:vol_thresh}
Let $\gamma_n \sim \rho n$ as $n \to \infty$ for some $\rho > 0$. Then,
\begin{align*}
\lim_{n \rightarrow \infty} \frac{1}{n} \ln \mathds{E}[V_n(Z_{0,n})] = - \ln \rho + \ln \pi - \frac{1}{2}.
\end{align*}
In addition,
\begin{align*}
\lim_{n \rightarrow \infty} \mathds{E}[V_n(Z_{0,n})] = \begin{cases} 0, & \rho > \frac{\pi}{\sqrt{e}} \\ \infty, & \rho < \frac{\pi}{\sqrt{e}}. \end{cases}
\end{align*}
\end{prop}

\begin{proof}
By \eqref{e:asymptotics}, as $n \to \infty$,
\begin{equation*}
(n!\kappa_n)^{1/n} \frac{n \kappa_n }{2\gamma \kappa_{n-1}} \sim \left( \sqrt{2\pi n}\left(\frac{n}{e}\right)^n \frac{1}{\sqrt{n\pi}}\left(\frac{2\pi e}{n}\right)^{n/2}\right)^{1/n}\frac{\sqrt{2\pi n}}{2\gamma}  \sim \frac{n}{e} \frac{\sqrt{2\pi e}}{\sqrt{n}}\frac{\sqrt{\pi n}}{\sqrt{2}\gamma}= \frac{\pi}{\sqrt{e}} \frac{n}{\gamma}.
\end{equation*}
Thus, by \eqref{e:vol}, under the assumption $\gamma_n \sim \rho n$, we have the following limiting behavior:
\begin{align*}
\lim_{n \rightarrow \infty} \frac{1}{n} \ln \mathds{E}[V_n(Z_{0,n})] = \lim_{n \rightarrow \infty}   \ln \left[(n!\kappa_n)^{1/n} \frac{n \kappa_n }{2\gamma \kappa_{n-1}} \right] = \ln \frac{\pi}{\sqrt{e}\rho}  = - \ln \rho + \ln \pi - \frac{1}{2}.
\end{align*}
This implies the last statement.
\end{proof}

Another viewpoint is to consider the volume of the typical cell $Z_n$ of the tessellation. This measures the volume of a typical collection of data that is compressed together. 

\begin{prop}\label{p:typ_vol}
If $\gamma_n \sim \rho n$ for some $\rho > 0$ as $n \to \infty$, then 
\begin{align*}
\lim_{n \rightarrow \infty} \frac{1}{n} \ln \mathds{E}[V_n(Z_n)] = - \ln \rho - \frac{1}{2}.
\end{align*}
In addition,
\[\lim_{n \to \infty} \mathds{E}[V_n(Z_n)] = \begin{cases} 0, & \rho > \frac{1}{\sqrt{e}} \\ \infty, & \rho < \frac{1}{\sqrt{e}}. \end{cases}\]
\end{prop}

\begin{proof}
By \eqref{e:typ_vol}, the expected value of the volume is
\begin{align*}
\mathds{E}[V_n(Z_n)] = \frac{1}{\kappa_n}\left( \frac{n \kappa_n }{ \gamma\kappa_{n-1}}\right)^n. 
\end{align*}
Then, by \eqref{e:asymptotics}, as $n \to \infty$,
\begin{align*}
\frac{1}{\kappa_n^{1/n}} \frac{n \kappa_n }{ \gamma_n\kappa_{n-1}}&\sim (n\pi)^{1/n}\left(\frac{n}{2\pi e}\right)^{1/2} \frac{\sqrt{2\pi n}}{\gamma_n} \sim \frac{n}{\gamma_n \sqrt{e}}.
\end{align*}
Thus, assuming $\gamma_n \sim \rho n$ as $n \to \infty$ for $\rho > 0$, 
\begin{align*}
\lim_{n \rightarrow \infty} \frac{1}{n} \log \mathds{E}[V(Z_n)] = - \log \rho - \frac{1}{2}.
\end{align*}
The right hand side is positive if $\rho < e^{-1/2}$ and negative if $\rho > e^{-1/2}$, which implies the last statement of the proposition.
\end{proof}

When the data set is (the support of) a stationary Poisson point process, the volume of the zero cell has to be replaced by the number of points of $N_n$ that lie in $Z_{0,n}$.
A similar threshold exists for the expected amount of data in $Z_{0,n}$, but it depends on the intensity of $N_n$. 
This then implies that for $\rho$ big enough, the probability 
that there is another data point in the cell of a typical data is small, meaning that with high probability, the cell of the tessellation determines the data uniquely. 

\begin{prop}\label{p:Poi_near}
For each $n$, assume $N_n$ is a Poisson point process in $\R^n$ with intensity
$\lambda_n = n^{n(\alpha - 1)} e^{n \lambda}$ for some $\lambda \in \R$ and $\alpha \in \R$.
Let 
$\gamma_n \sim \rho n^{\alpha}$ as $n \to \infty$ for some $\rho > 0$.  Then,
\[ \lim_{n \rightarrow \infty} \mathds{E}^{0,!}_n[N_n(Z_{0,n})] = \begin{cases} 0, &\rho > {e^{\lambda}\pi}/{\sqrt{e}}\\ \infty, & \rho > e^{\lambda}\pi/ \sqrt{e} \end{cases}.\]
Thus, for $\rho > \frac{e^{\lambda}\pi}{\sqrt{e}}$, 
\[\lim_{n \rightarrow \infty} \mathds{P}_{n}^0(N_n(Z_{0,n}) = 1) \rightarrow 1.\]
\end{prop}

\begin{proof}
By Slivnyak's theorem,
\begin{align}\label{e:exp_N}
\mathds{E}^{0, !}_n[N_n(Z_{0,n})] &= \mathds{E}[N_n(Z_{0,n})] =  \mathds{E}\left[ \mathds{E}[N_n(Z_{0,n}) |Z_{0,n}] \right]
= \lambda_n \mathds{E}[V_n(Z_{0,n})].
\end{align}

By the assumption on $\gamma_n$ and \eqref{e:asymptotics},
\begin{align}\label{e:c_n}
\frac{2\gamma_n \kappa_{n-1}}{n \kappa_n} \sim \frac{\sqrt{2}\rho}{\sqrt{\pi}} n^{\alpha - \frac{1}{2}}, \text{ as } n \to \infty.
\end{align}
Then, by \eqref{e:asymptotics} and \eqref{e:c_n}, as $n \rightarrow \infty$,
\begin{equation*}
\frac{1}{n} \log \mathds{E}[V_n(Z_{0,n})] \sim \log(\frac{n}{e}) + \frac{1}{2}\log\frac{2\pi e}{n} + \log\frac{\sqrt{\pi}}{\sqrt{2}\rho n^{\alpha - \frac{1}{2}}} = (1 - \alpha) \log n + \log \frac{\pi}{\rho\sqrt{e}}.
\end{equation*}
By the assumption on $\lambda_n$ and \eqref{e:exp_N}, as $n \to \infty$,
\begin{equation}\label{e:exp_N_asymp}
\frac{1}{n} \log \mathds{E}^{0,!}_n[N_n(Z_{0,n})] \sim \lambda + \log \frac{\pi}{\rho \sqrt{e}}. 
\end{equation}
The threshold follows. Then, by Slivnyak's theorem and Jensen's inequality,
\[\mathds{P}_{n}^0(N_n(Z_{0,n}) = 1) = \mathds{P}(N_n(Z_{0,n}) = 0) = \E[e^{-\lambda_n V_n(Z_{0,n})}] \geq e^{-\lambda_n \mathds{E}[V_n(Z_{0,n})]} = e^{-\mathds{E}^{0,!}_n[N_n(Z_{0,n})]} .\]
Thus, for $\rho > {e^{\lambda}\pi}/{\sqrt{e}}$,
\[\lim_{n \to \infty} \mathds{P}_{n}^0(N_n(Z_{0,n}) = 1) = 1.\]
\end{proof}


\subsection{Farthest distance between two data points compressed together}


Another and more demanding low distortion criterion is that all the data compressed together be close in Euclidean distance. Consider first the case when the data is all of $\R^n$. We want to find the scaling necessary for $\gamma_n$ to ensure that all data points in the zero cell are within some distance from the typical data point at the origin. This is equivalent to showing that the radius of the smallest ball centered at the origin that contains all of the zero cell is small. As mentioned in Section \ref{s:zero}, a closed form for the distribution of this radius $R_M$ is only known in dimension two, but we can obtain bounds that give the following asymptotic behavior.

\begin{theorem}\label{t:verts}
Assume $\gamma_n \sim \rho n^{\alpha}$ as $n \to \infty$ and let $R > 0$. Then, there exists $\rho_u > \frac{\sqrt{\pi}}{R\sqrt{2}}$ such that for all $\rho > \rho_u$,
\[\lim_{n \to \infty} \PP(R_M(Z_{0,n}) \geq n^{3/2 - \alpha}R) = 0.\]
Also, there exists $\rho_{\ell} < \frac{\sqrt{\pi}}{R\sqrt{2}}$ such that for all for $\rho < \rho_{\ell}$,
\[\lim_{n \to \infty} \PP(R_M(Z_{0,n}) \leq n^{3/2 - \alpha}R) = 0.\]
\end{theorem} 

Before proving the Theorem, we need the following. Define the beta prime density with parameters $n \in \mathbb N$ and $\sigma>0$ as follows:
\[f_{n, \sigma}(x) = c_{n, \sigma}\left(1 + \frac{|x|^2}{\sigma^2}\right)^{-\frac{n+1}{2}} \text{  for  } x \in \R^n, \text{ with } c_{n, \sigma} = \frac{\Gamma(\frac{n+1}{2})}{\sigma^n\pi^{n/2} \Gamma( \frac{1}{2})}.\]
Let $X_1, \ldots, X_m$ be i.i.d random vectors in $\R^n$ with density $f_{n, \sigma}$ and let $P^{\sigma}_{m,n}$ denote the convex hull of these points. Also, define $A := A(X_1, ..., X_n)$ to be the $d-1$ dimensional affine subspace containing the points $X_1, \ldots, X_n$, and let $h(A)$ be the signed distance from the origin to the subspace $A$.  The following lemma gives the probability that the points $X_1, \ldots, X_n$ form a face of $P^{\sigma}_{m,n}$.

\begin{lemma}\label{l:facet}
\begin{equation*}
\begin{aligned}
&\PP\left([X_1, \ldots, X_n] \text{ is a face in } P_{m, n}^{ \sigma} \text{ such that } |h(A)| \leq r\right) \\
& \qquad = \frac{2\Gamma(\frac{d + 1}{2})}{\sigma\Gamma(\frac{ d}{2})\sqrt{\pi}} \int_{-r}^{r} \left(1+\frac{t^2}{\sigma^2}\right)^{-\frac{ n + 1}{2}} \left( \frac{1}{\sigma \pi} \int_{-\infty}^t \left(1+\frac{s^2}{\sigma^2}\right)^{-1} \dint s \right)^{m-n} \dint t. 
\end{aligned}
\end{equation*}
\end{lemma}

\begin{proof}
Let $\pi_{A^{\perp}}$ be the projection from $\R^n$ to the 1-dimensional subspace $A^{\perp}$ and define the isometry $I_{A^{\perp}}: A^{\perp} \mapsto \R$ such that $I_{A^{\perp}}(0) = 0$.

By Lemma 3.1 in \cite{thale18}, if $X$ has density $f_{n, \sigma}$, then $I_{A^{\perp}}(\pi_{A^{\perp}}(X))$ has density 
\[f_{1, \sigma}(s) = \frac{1}{\sigma \pi }\left(1 + \frac{s^2}{\sigma^2}\right)^{-1}.\]
 This was stated with $\sigma = 1$ in the reference, but if $X$ has density $f_{n, \sigma}$, then $X/\sigma$ has density $\tilde{f}_{n, 1}$, and the more general statement follows from a change a variables, since $I_{A^{\perp}}(\pi_{A^{\perp}}(X/\sigma)) = I_{A^{\perp}}(\pi_{A^{\perp}}(X))/\sigma$. 

 Also, by Corollary 3.6 in \cite{thale18}, if $X_1, \ldots X_n$ have  the beta prime density $f_{n, 1}$, then $h^2(A)/\sigma^2$ has density 
\[g(t) = \frac{\Gamma(\frac{n + 1}{2})}{\Gamma(\frac{n}{2})\sqrt{\pi}} t^{-\frac{1}{2}}(1 + t)^{-(\frac{n + 1}{2})}1_{\{t \geq 0\}}.\]
By a changes of variables,
\[\PP\left(|h(A)| \leq r \right) = \frac{2\Gamma(\frac{n + 1}{2})}{\Gamma(\frac{n}{2})\sqrt{\pi}} \int_0^{r/\sigma} (1 + y^2)^{-\frac{n + 1}{2}}dy = \frac{2\Gamma(\frac{n + 1}{2})}{\sigma\Gamma(\frac{n}{2})\sqrt{\pi}} \int_0^{r} \left(1 + \frac{t^2}{\sigma^2}\right)^{-\frac{n + 1}{2}}dt.\]
Hence, the distribution of $|h(A)|$ has density 
\[\tilde{h}(t) = \frac{2\Gamma(\frac{n + 1}{2})}{\sigma\Gamma(\frac{n}{2})\sqrt{\pi}}\left(1 + \frac{t^2}{\sigma^2}\right)^{-\frac{n+ 1}{2}}1_{\{t\geq 0\}}.\] 
Then, by the fact that $[X_1, \ldots, X_n]$ is a facet of $P_{m,n}^{\sigma}$ if and only if $I_{A^{\perp}}(\pi_{A^{\perp}}(X_i)) \leq h(A)$ for all $i = n+1, \ldots, m$, or $I_{A^{\perp}}(\pi_{A^{\perp}}(X_i)) \geq h(A)$ for all $i = n+1, \ldots, m$. This gives
\begin{align*}
&\mathds{P}\left([X_1, \ldots, X_n] \text{ is a facet in } P_{m,n}^{ \sigma} \text{ such that } |h(A)| \leq r \right) \\
&= \int_{0}^r \mathds{P}\left([X_1, \ldots, X_n] \text{ is a facet } \text{ in }P_{m, n}^{ \sigma} \, \bigg| \, |h(A)| = t \right) \tilde{h}(t) dt \\
&= \int_{0}^r (\mathds{P}\left( I_{A^{\perp}}(\pi_{A^{\perp}}(X_i)) \leq t \text{ for each } i = n+1, \ldots, m \right) \\
&\qquad \qquad + \mathds{P}\left( I_{A^{\perp}}(\pi_{A^{\perp}}(X_i)) \geq t \text{ for each } i = n+1, \ldots, m \right)) \tilde{h}(t) dt \\
&= \int_{0}^r \left(\int_{-\infty}^t f_{1,\sigma}(s) ds \right)^{m - n} \tilde{h}(t) dt + \int_{0}^r \left(\int_{t}^{\infty} \tilde{f}_{1,\sigma}(s) ds \right)^{m - n} \tilde{h}(t) dt \\
&= \int_{-r}^r \left(\int_{-\infty}^t f_{1,\sigma}(s) ds \right)^{m - n} \tilde{h}(t) dt, 
\end{align*}
where the last inequality follows from the fact that the densities are symmetric. Hence,
\begin{align*}
&\mathds{P}\left([X_1, \ldots, X_n] \text{ is a facet in } P_{m, n}^{ \sigma} \text{ such that } |h(A)| \leq r \right) \\
&\qquad =  \frac{2\Gamma(\frac{n + 1}{2})}{\sigma\Gamma(\frac{n}{2})\sqrt{\pi}}\int_{-h}^{h} \left(1+\frac{t^2}{\sigma^2}\right)^{-\frac{n+ 1}{2}} \left((\sigma\pi)^{-1} \int_{-\infty}^t \left(1+\frac{s^2}{\sigma^2}\right)^{-1} \dint s \right)^{m-n} \dint t. 
\end{align*}
\end{proof}

We can now prove Theorem \ref{t:verts}.
\begin{proof} (of Theorem \ref{t:verts})

Let $X$ be a random vector in $\R^n$ with density $f_{n, \sigma}$.
By a generalization of Lemma 7.7 in \cite{cones}, we have the vague convergence
\begin{equation}\label{e:vcon}
m\PP(m^{-1} X \in \cdot )  \to \nu(\cdot),
\end{equation}
as $m \to \infty$, where $\nu$ is a measure on $\R^n\backslash \{0\}$ with density 
\begin{equation}
\label{eq:s-density}
x \mapsto \frac{2{\sigma}}{\omega_{n + 1}} |x|^{-n - 1}.
\end{equation}

Let $\Pi_{n}(\sigma)$ be a Poisson point process on $\R^n\backslash \{0\}$ with intensity measure $\nu$. 
Then, \eqref{e:vcon} implies the following generalization of $(4.6)$ in \cite{cones}: As $m \to \infty$,
\begin{equation}\label{e:beta_conv}
\sum_{i=1}^{m} \delta_{X_i/m} \to \Pi_{n}(\sigma) \text{  in distribution},
\end{equation}
where $X_1, \ldots, X_m$ are i.i.d random vectors in $\R^n$ with density $f_{n, \sigma}$. 
Now, let $P^{\sigma}_{m,n}$ be the convex hull of $X_1, \ldots, X_m$. The convergence \eqref{e:beta_conv} implies that
\begin{align*}
&\lim_{m \to \infty} \E[ \# \text{ of faces within distance } mh \text{ in } P_{m, n}^{\sigma} ]
\\
& \qquad = \E[ \# \text{ of faces within distance } h \text{ in } \mathcal{C}(\Pi_{n}(\sigma))],
\end{align*}
with ${\mathcal C}(P)$ denoting the convex hull of the points in set $P$.
Now, by the same argument as in the proof of Theorem 1.21 of \cite{thale18}, the convex dual of ${\mathcal C}(\Pi_{n}(\sigma))$ has the same distribution as the zero cell $Z_{0,n}$ of a stationary and isotropic hyperplane tessellation with intensity $\gamma_n = \frac{\sigma \omega_{n}}{\omega_{n+1}}$.
Hence, the distances to the faces of the convex hull of $\Pi_{n}(\sigma)$ are the reciprocal of the distances to the vertices of $Z_{0, n}$. This gives
\begin{align*}
&\mathds{E}[ \# \text{ of vertices at distance greater than }r\text{ in }Z_{0, n}] \\
& \qquad= \mathds{E}[\# \text{ of faces at distance less than } r^{-1} \text{ in } {\mathcal C}(\Pi_{n}(\sigma))] \\
&\qquad =\lim_{m \to \infty} \E[ \# \text{ of faces at distance less than } mr^{-1} \text{ in } P_{m,n}^{\sigma} ] \\
&\qquad =  \lim_{m \to \infty}  \binom{m}{n} \mathds{P}\left([X_1, \ldots, X_n] \text{ is a face of } P_{m, n}^{\sigma} \text{ such that }|h(A)| \leq mr^{-1} \right) \\
& \qquad =  \lim_{m \to \infty}  \binom{m}{n} \frac{2\Gamma(\frac{n + 1}{2})}{\sqrt{\pi} \Gamma(\frac{n}{2})} \int_{-m/r}^{m/r} \left(1+\frac{t^2}{\sigma^2}\right)^{-\frac{n + 1}{2}} \left( \frac{1}{\pi \sigma} \int_{-\infty}^t \left(1+\frac{s^2}{\sigma^2}\right)^{-1} \dint s \right)^{m-n} \dint t,
\end{align*}
where the last equality follows from Lemma \ref{l:facet}.
By the same arguments as in Lemma 4.9 in \cite{thale18}, as $m \to \infty$,
\begin{align*}
\int_{-m/r}^{m/r} \left(1+\frac{t^2}{\sigma^2}\right)^{-\frac{n + 1}{2}} \left( \frac{1}{\pi \sigma} \int_{-\infty}^t \left(1+\frac{s^2}{\sigma^2}\right)^{-1} \dint s \right)^{m-n} \dint t \sim m^{-n} \sigma \pi^{n} \Gamma(n)\Gamma_u\left(n, \pi^{-1} \sigma r\right).
\end{align*}
Then, since $\binom{m}{n} \sim \frac{m^n}{n!}$ as $m \to \infty$,
\begin{align*}
& \binom{m}{n} \frac{2\Gamma(\frac{n + 1}{2})}{\sigma\sqrt{\pi} \Gamma(\frac{n}{2})}\int_{-m/r}^{m/r} (1+t^2)^{-\frac{n + 1}{2}} \left( \tilde{c}_{1, \frac{n + 1}{2}} \int_{-\infty}^t (1+s^2)^{-1} \dint s \right)^{m-n} \dint t \\
 & \sim \frac{2}{n} \frac{\Gamma(\frac{n + 1}{2})\pi^{n}}{\sqrt{\pi} \Gamma(\frac{n}{2})}  \Gamma_u\left(n, \sigma \pi^{-1} r\right) 
 = \pi^{n - \frac{1}{2}} \frac{\Gamma(\frac{n + 1}{2})}{\Gamma(\frac{n}{2} + 1)}\Gamma_u\left(n, \sigma \pi^{-1}r\right).
\end{align*}
Let $\gamma_n = \frac{\sigma \omega_{n}}{\omega_{n+1}}$, i.e., let $\sigma = \gamma_n  \frac{\omega_{n+1}}{\omega_{n}}$.
Then,
\[\mathds{E}[ \# \text{ of vertices farther than } r \text{ in }Z_{0, n}] =  \frac{\Gamma(\frac{n + 1}{2})\pi^{n}}{\sqrt{\pi} \Gamma(\frac{n}{2} + 1)} \Gamma_u\left(n, \gamma_n\frac{\omega_{n+1}}{\pi \omega_{n}} r\right).\]
Similar computations give
\[\mathds{E}[ \# \text{ of vertices closer than } r \text{ in }Z_{0, n}] =  \frac{\Gamma(\frac{n + 1}{2})\pi^{n} }{\sqrt{\pi} \Gamma(\frac{n}{2} + 1)} \Gamma_{\ell}\left(n, \gamma_n\frac{\omega_{n+1}}{\pi \omega_{n}} r\right).\]
Now, by Markov's inequality,
\begin{align*}
\PP(R_M(Z_{0,n}) \geq n^{3/2 - \alpha}R) &= \PP(\# \text{ of vertices farther than } n^{3/2 - \alpha}R \text{ in }Z_{0, n} > 0)\\
& \leq \E[ \# \text{ of vertices farther than }  n^{3/2 - \alpha}R  \text{ in }Z_{0, n}] \\
&= \frac{\Gamma(\frac{n + 1}{2})\pi^{n} }{\sqrt{\pi} \Gamma(\frac{n}{2} + 1)} \Gamma_{u}\left(n, \gamma_n\frac{\omega_{n+1}}{\pi \omega_{n}} n^{3/2 - \alpha}R \right). 
\end{align*}
Also, 
\begin{align*}
\PP(R_M(Z_{0,n}) \leq n^{3/2 - \alpha}R) &\leq \PP(\# \text{ of vertices closer than } n^{3/2 - \alpha}R \text{ in }Z_{0, n}  > 0)\\
& \leq \E[ \# \text{ of vertices closer than }  n^{3/2 - \alpha}R  \text{ in }Z_{0, n}] \\
&=\frac{\Gamma(\frac{n + 1}{2})\pi^{n}}{\sqrt{\pi} \Gamma(\frac{n}{2} + 1)} \Gamma_{\ell}\left(n, \gamma_n\frac{\omega_{n+1}}{\pi \omega_{n}} n^{3/2 - \alpha}R \right).
\end{align*}
By the assumption on $\gamma_n$ and \eqref{e:gamma}, as $n \to \infty$,
\begin{equation*}
\gamma_n\frac{\omega_{n+1}}{\pi \omega_{n}}n^{3/2 - \alpha}R  \sim \rho n^{\alpha} \frac{2\pi^{\frac{n+1}{2}} \Gamma(n/2)}{\pi \Gamma(\frac{n+1}{2}) 2\pi^{n/2}} n^{3/2 - \alpha}R \sim \rho n^{3/2}\left(\frac{2}{\pi n}\right)^{1/2}R = \frac{\rho R\sqrt{2}}{\sqrt{\pi}} n.
\end{equation*}
Then, by Laplace's method (see Lemma A.2 in \cite{Oreilly}), for $\rho > \frac{\sqrt{\pi}}{R\sqrt{2}}$,
\begin{equation*}
\lim_{n \to \infty} \frac{1}{n} \ln \Gamma_{u}\left(n, \gamma_n\frac{\omega_{n+1}}{\pi \omega_{n}} n^{3/2 - \alpha}R \right) = \ln  \frac{\rho R\sqrt{2}}{\sqrt{\pi}} -  \frac{\rho R\sqrt{2}}{\sqrt{\pi}} + 1.
\end{equation*}
and similarly, for $\rho < \frac{\sqrt{\pi}}{R\sqrt{2}}$,
\begin{equation*}
\lim_{n \to \infty} \frac{1}{n} \ln \Gamma_{\ell}\left(n, \gamma_n\frac{\omega_{n+1}}{\pi \omega_{n}} n^{3/2 - \alpha}R \right) = \ln  \frac{\rho R\sqrt{2}}{\sqrt{\pi}} -  \frac{\rho R\sqrt{2}}{\sqrt{\pi}} + 1.
\end{equation*}
Since $\frac{\Gamma(\frac{n + 1}{2})}{\sqrt{\pi} \Gamma(\frac{n}{2} + 1)} = O(n^{-1/2})$, for $\rho > \frac{\sqrt{\pi}}{R\sqrt{2}}$,
\begin{equation*}
\lim_{n \rightarrow \infty} \frac{1}{n} \ln \PP(R_M(Z_{0,n}) \geq n^{3/2 - \alpha}R) \leq  \ln \rho R\sqrt{2\pi}-  \frac{\rho R\sqrt{2}}{\sqrt{\pi}} + 1,
\end{equation*}
and for $\rho < \frac{\sqrt{\pi}}{R\sqrt{2}}$,
\begin{equation*}
\lim_{n \rightarrow \infty} \frac{1}{n} \ln \PP(R_M(Z_{0,n}) \leq n^{3/2 - \alpha}R) \leq  \ln \rho R\sqrt{2\pi}-  \frac{\rho R\sqrt{2}}{\sqrt{\pi}} + 1,
\end{equation*}
The function $ \ln \pi + \ln x - x + 1$ is concave, and has two zeros, one $0 < x_{\ell} < 1$
and one where $x_u > 1$. These zeros determine the values of
$$\rho_{\ell}:= x_{\ell} \frac{\sqrt{\pi}}{R\sqrt{2}}  \quad \mbox{and}\quad \rho_u:=
x_u \frac{\sqrt{\pi}}{R\sqrt{2}},$$
respectively.

\end{proof}

Next consider the case where the underlying data is a Poisson point process, and more precisely the regime where the expected number of points in the zero cell goes to infinity. 
Theorem \ref{t:main_poi} below gives a sufficient condition for all points of the point process which are contained in the zero cell (the cell of the typical data) to be within distance $R_n$ from the point at the origin (the typical data). The result also shows that the same scaling that is sufficient for the criterion to be satisfied is also necessary. 


\begin{theorem}\label{t:main_poi}
Consider the setting of the Proposition \ref{p:Poi_near}, with $\lambda$ fixed, and assume that $\rho <\rho^*:= \frac{e^{\lambda} \pi}{\sqrt{e}}$.
\begin{enumerate}
\item[(i)] If $R> \frac {\sqrt{e}} {e^\lambda \sqrt{2\pi}}$,
then $\frac{\sqrt{\pi}}{R\sqrt{2}} <\rho^*$ and for all $\rho$ in the interval
$(\frac{\sqrt{\pi}}{R\sqrt{2}},\rho^*)$,
\begin{eqnarray*}
& & \hspace{-1cm}\limsup_{n \rightarrow \infty} \frac{1}{n} \log \mathds{P}_n^0\left( \max_{x_i \in N_n \cap Z_{0,n}} |x_i| \geq Rn^{\frac{3}{2} - \alpha} \, \bigg| \, N_n(Z_{0,n}) > 0 \right) \\
& & \leq \lambda + \frac{1}{2}\log 2\pi e + \log R - \frac{\sqrt{2}\rho R}{\sqrt{\pi}} + \log 4.
\end{eqnarray*}
\item[(ii)] Let 
$$a(R,\lambda)= \max \left(\left(\lambda + \frac{1}{2}\log 2\pi e + \log R + \log 4\right),1\right)\ge 1.$$
If $R$ is such that 
$\rho_u:=\frac{\sqrt{\pi}}{R\sqrt{2}}\ a(R,\lambda) <\rho^*$, which holds for $R$ large enough, then
for all $\rho$ in the interval $(\rho_u,\rho^*)$,
\begin{align}\label{e:main_gR}
\lim_{n \to \infty} \mathds{P}_n^0\left( \max_{x_i \in N_n \cap Z_0} |x_i| \geq Rn^{\frac{3}{2} - \alpha} \, \bigg| \, N_n(Z_{0,n}) > 0 \right) = 0,
\end{align}
where the convergence is at least exponential of rate 
$\lambda + \frac{1}{2}\log 32\pi eR^2 - \frac{\sqrt{2}\rho R}{\sqrt{\pi}} < 0$.
\item[(iii)] 
For all $\rho < \min\left(\frac{\sqrt{\pi}}{R\sqrt{2}}, \rho^*\right)$,
\begin{eqnarray*}
& & \hspace{-1cm}\limsup_{n \rightarrow \infty} \frac{1}{n} 
\log \mathds{P}_n^0\left( \max_{x_i \in N_n \cap Z_{0,n}} |x_i| \leq n^{\frac{3}{2} - \alpha}R \,\bigg| \, N_n(Z_{0,n})  > 0 \right)
\\ && \leq \lambda + \frac{1}{2}\log 2\pi e + \log R - \frac{\sqrt{2}\rho R}{\sqrt{\pi}} + \log 4.
\end{eqnarray*}
\item[(iv)] 
If $R < (4e^{\lambda}\sqrt{2\pi e})^{-1}$, then 
for all $\rho$ in $\left(0, \min\left(\frac{\sqrt{\pi}}{R\sqrt{2}}, \rho^* \right) \right)$,
\begin{align}\label{e:main_lR}
\lim_{n \to \infty} \mathds{P}_n^0\left( \max_{x_i \in N_n \cap Z_0} |x_i| \leq n^{\frac{3}{2} - \alpha} R \, \bigg| \, N_n(Z_{0,n}) > 0 \right) = 0,
\end{align}
where the convergence is at least exponential of rate 
$\lambda + \frac{1}{2}\log 32\pi eR^2 - \frac{\sqrt{2}\rho R}{\sqrt{\pi}} < 0$.
\end{enumerate}
\end{theorem}

\begin{proof}
First, by \eqref{e:x_in_Z0} and two changes of variable,
\begin{align*}
\mathds{E}[N_n(Z_{0,n} \cap B_n(R)^c)] &=  \lambda_n \mathds{E}[V_n(Z_{0,n} \cap B_n(R)^c)]  \\
&= \lambda_n \int_{B_n(R)^c} \mathds{P}(x \in Z_{0,n}) dx  = \int_{B_n(R)^c} e^{- \frac{2\gamma_n\kappa_{n-1}}{n \kappa_n}  |x|} dx \\
&= \lambda_n n \kappa_n \int_R^{\infty} r^{n-1} e^{- \frac{2\gamma_n\kappa_{n-1}}{n \kappa_n} r} dr \\
&=  n\kappa_n \left(\frac{2\gamma_n\kappa_{n-1}}{n \kappa_n} \right)^{-n} \int_{\frac{2\gamma_n\kappa_{n-1}}{n \kappa_n} R}^{\infty} y^{n-1} e^{- y} dy \\
&= \lambda_n n!\kappa_n \left(\frac{2\gamma_n\kappa_{n-1}}{n \kappa_n} \right)^{-n} \Gamma_u\left(n, \frac{2\gamma_n\kappa_{n-1}}{n \kappa_n} R\right) \\
&= \lambda_n \mathds{E}[V_n(Z_{0,n} )] \Gamma_u\left(n, \frac{2\gamma_n\kappa_{n-1}}{n \kappa_n} R\right) \\
&= \mathds{E}[N_n(Z_{0,n})]\Gamma_u\left(n, \frac{2\gamma_n\kappa_{n-1}}{n \kappa_n} R\right),
\end{align*}
and similarly,
\[ \mathds{E}[N_n(Z_{0,n} \cap B_n(R))]  = \mathds{E}[N_n(Z_{0,n} )] \Gamma_{\ell}\left(n, \frac{2\gamma_n\kappa_{n-1}}{n \kappa_n} R\right).\]
By Laplace's method (see Lemma A.2 in \cite{Oreilly}) and \eqref{e:c_n}, if $\frac{\sqrt{2}\rho R}{\sqrt{\pi}} > 1$,
\begin{equation*}
\lim_{n \to \infty} \frac{1}{n} \ln \Gamma_{u}\left(n, \frac{2\gamma_n\kappa_{n-1}}{n \kappa_n} n^{\frac{3}{2} - \alpha}R\right) = \log \frac{\sqrt{2}\rho R}{\sqrt{\pi}} -  \frac{\sqrt{2}\rho R}{\sqrt{\pi}} + 1,
\end{equation*}
and the limit is 1 otherwise. Then, by \eqref{e:exp_N_asymp},
\begin{equation}\label{e:ZBRc}
\lim_{n \to \infty} \frac{1}{n} \ln \mathds{E}[N_n(Z_{0,n} \cap B_n(n^{\frac{3}{2} - \alpha}R)^c)] = \begin{cases} \lambda + \ln \sqrt{2\pi e}R -  \frac{\sqrt{2}\rho R}{\sqrt{\pi}}, &\frac{\sqrt{2}\rho R}{\sqrt{\pi}} > 1\\ \lambda + \ln \frac{\pi}{\rho \sqrt{e}}, & \frac{\sqrt{2}\rho R}{\sqrt{\pi}} < 1.\end{cases}
\end{equation}
Similarly,
\begin{equation}\label{e:ZBR}
\lim_{n \to \infty} \frac{1}{n} \ln \mathds{E}[N_n(Z_{0,n} \cap B_n(n^{\frac{3}{2} - \alpha}R))] = \begin{cases} \lambda + \ln \sqrt{2\pi e}R  -  \frac{\sqrt{2}\rho R}{\sqrt{\pi}}, &\frac{\sqrt{2}\rho R}{\sqrt{\pi}} < 1\\ \lambda + \ln \frac{\pi}{\rho \sqrt{e}}, & \frac{\sqrt{2}\rho R}{\sqrt{\pi}} > 1.\end{cases}
\end{equation} 

Next, by the fact that $N_n$ is Poisson,
\[ \mathds{E}[N_n(Z_{0,n})^2] = \E[ \mathds{E}[N_n(Z_{0,n})^2 | Z_{0,n}]] = \lambda_n^2\mathds{E}[V_n(Z_{0,n})^2] + \lambda_n\mathds{E}[V_n(Z_{0,n})],\]
and by the second moment inequality, we have
\begin{align}\label{e:NZ_lb}
\mathds{P}(N_n(Z_{0,n}) > 0) &\geq \frac{\mathds{E}[N_n(Z_{0,n})]^2}{\mathds{E}[N_n(Z_{0,n})^2]} = \frac{\lambda_n^2\mathds{E}[V_n(Z_{0,n})]^2}{\lambda_n^2\mathds{E}[V_n(Z_{0,n})^2] + \lambda_n\mathds{E}[V_n(Z_{0,n})]} \nonumber \\
&= \frac{\mathds{E}[V_n(Z_{0,n})]^2}{\mathds{E}[V_n(Z_{0,n})^2]} \left( \frac{1}{1 + \frac{\lambda_n\mathds{E}[V_n(Z_{0,n})]}{\lambda_n^2\mathds{E}[V_n(Z_{0,n})^2]}} \right). 
\end{align}
Then, by Jensen's inequality,
\begin{align*}
\frac{\lambda_n\mathds{E}[V_n(Z_{0,n})]}{\lambda_n^2\mathds{E}[V_n(Z_{0,n})^2]} &\leq \frac{\lambda_n\mathds{E}[V_n(Z_{0,n})]}{\lambda_n^2\mathds{E}[V_n(Z_{0,n})]^2} = \frac{1}{\lambda_n\mathds{E}[V_n(Z_{0,n})]},
\end{align*}
and by the assumption on $\rho$, $\lim_{n \rightarrow \infty} \lambda_n\mathds{E}[V_n(Z_{0,n})] = \infty$ by Proposition \ref{p:Poi_near}, and so
\[\lim_{n \to \infty} \frac{\lambda_n\mathds{E}[V_n(Z_{0,n})]}{\lambda_n^2\mathds{E}[V_n(Z_{0,n})^2]} = 0.\]
Then, by \eqref{e:vol} and \eqref{e:Vmombnd}, as $n \to \infty$,
\[\mathds{P}(N_n(Z_{0,n}) > 0) \gtrsim  \frac{\mathds{E}[V_n(Z_{0,n})]^2}{\mathds{E}[V_n(Z_{0,n})^2]} \sim \frac{\Gamma(n + 1)^2}{\Gamma(2n + 1)}.\]
Now, by Markov's inequality and \eqref{e:NZ_lb},
\begin{align*}
&\mathds{P}_n^0\left( \max_{x_i \in Z_{0,n} \cap N_n} |x_i| \geq n^{\frac{3}{2} - \alpha}R\, \bigg| \, N_n(Z_{0,n})  > 0 \right) \\
&\qquad = \mathds{P}_n^{0, !}\left( N_n(Z_{0,n} \cap B_n(n^{\frac{3}{2} - \alpha}R)^c) > 0 \, \bigg| \, N_n(Z_{0,n})  > 0 \right) \\
&\qquad = \frac{\mathds{P}_n^{0,!}\left( N_n(Z_{0,n}\cap B_n(n^{\frac{3}{2} - \alpha}R)^c)  > 0 \right)}{\mathds{P}^{0,!}_n(N_n(Z_{0,n}) > 0)}  \\
&\qquad \lesssim \mathds{E}[N_n(Z_{0,n} \cap B_n(n^{\frac{3}{2} - \alpha}R)^c)] \frac{\Gamma(2n + 1)}{\Gamma(n+1)2}.
\end{align*}
Thus, by \eqref{e:asymptotics} and \eqref{e:ZBRc}, for $\rho > \frac{\sqrt{\pi}}{R\sqrt{2}}$,
\begin{equation*}
\limsup_{n \rightarrow \infty} \frac{1}{n} \ln \mathds{P}_n^0\left( \max_{x_i \in N_n \cap Z_{0,n}} |x_i| \geq R \, \bigg| \, N_n(Z_{0,n})  > 0 \right) \leq \lambda + \frac{1}{2}\log 2\pi e + \log R - \frac{\sqrt{2}\rho R}{\sqrt{\pi}} + \log 4.
\end{equation*}
Thus, for all $\rho > \rho_u := \max\{\frac{\sqrt{\pi}}{R\sqrt{2}}, \frac{\sqrt{\pi}}{R\sqrt{2}}(\lambda + \frac{1}{2}\log 2\pi e + \log R + \log 4)\}$, 
\[\lim_{n \rightarrow \infty}  \mathds{P}_n^0\left( \max_{x_i \in N_n \cap Z_{0,n}} |x_i| \geq R \, \bigg| \, N_n(Z_{0,n})  > 0 \right) = 0.\]
This completes the proofs of (i) and (ii).

Now, again by Markov's inequality and \eqref{e:NZ_lb}, 
\begin{align*}
&\mathds{P}_n^0\left( \max_{x_i \in N_n \cap Z_{0,n}} |x_i| \leq n^{\frac{3}{2} - \alpha}R \, \bigg| \, N_n(Z_{0,n})  > 0 \right) \\
&\qquad \lesssim \mathds{E}[N_n(Z_{0,n} \cap B_n(n^{\frac{3}{2} - \alpha}R))] \frac{\Gamma(2n + 1)}{\Gamma(n+1)^2}.
\end{align*}
By \eqref{e:asymptotics} and \eqref{e:ZBR}, for $\rho < \frac{\sqrt{\pi}}{R\sqrt{2}}$,
\begin{align*}
&\limsup_{n \rightarrow \infty} \frac{1}{n} \log \mathds{P}_n^0\left( \max_{x_i \in N_n \cap Z_{0,n}} |x_i| \leq n^{\frac{3}{2} - \alpha} R \, \bigg| \, N_n(Z_{0,n})  > 0 \right) \\
&\qquad  \leq  \lambda + \frac{1}{2}\log 2\pi e + \log R - \frac{\sqrt{2}\rho R}{\sqrt{\pi}}+ \log 4.
\end{align*}
Thus, if $R < (4e^{\lambda}\sqrt{2\pi})^{-1}$ then 
for all $\rho < \frac{\sqrt{\pi}}{R\sqrt{2}}$, 
\[\lim_{n \rightarrow \infty}  \mathds{P}_n^0\left( \max_{x_i \in N_n \cap Z_{0,n}} |x_i| \leq n^{\frac{3}{2} - \alpha} R \, \bigg| \, N_n(Z_{0,n})  > 0 \right) = 0.\]
This completes the proofs of (iii) and (iv).
\end{proof}

\begin{rem}
To separate data more efficiently, we would ideally like to assume a relationship between $\lambda_n$ and $\gamma_n$ such that the cells of the tessellation contain more than one point with high probability. The assumption that 
$\lim_{n \rightarrow \infty} \mathds{E}^{0,!}_n[N_n(Z_{0,n})] = \infty$
does not ensure that $\lim_{n \rightarrow \infty} \mathds{P}^0_n(N_n(Z_{0,n}) > 1) = 1$, however.
The second moment method does not help, since this lower bound goes to zero as $n$ goes to infinity for all $\lambda_n$, and thus it remains an open question what scaling of $\lambda_n$ and $\gamma_n$ is needed to ensure $\lim_{n \rightarrow \infty} \mathds{P}^0_n(N_n(Z_{0,n}) > 1) = 1$.
\end{rem}



\section{Summary} 

Our results can be summarized in terms of phenomena that successively take place
when increasing $\rho$ for a given $\alpha$ and incrementing $\alpha$, when parameterizing
the intensity of hyperplanes as $\rho n^\alpha$. As soon as $\alpha$ is positive, 
one finds a data arbitrarily close and encoded differently w.h.p.
In addition, a displacement of order $\sqrt{n}$ in a random direction leads to
an encoding which is different w.h.p. When moving to $\alpha > \frac 1 2$, 
a displacement of order one in a random direction leads to an encoding which is different w.h.p.
Further phenomena start appearing when $\alpha=1$ (Shannon regime). When increasing $\rho$, one first
gets a small volume for the typical cell, and then for the zero cell w.h.p. 
At this scale, one can also control distortion, namely the fact that the most distant
data point encoded like the typical data is at distance at most $\sqrt{n}R$ w.h.p. by a proper
choice of $\rho$ with $\rho$ arbitrarily small as $R$ grows.
A new phenomenon appears at $\alpha=\frac{3}{2}$ where a sufficiently large $\rho$
guarantees that the most distant data point encoded like the typical data is at distance at most $R$ w.h.p. 
The following table illustrates how and when this collection of phenomena take place when
increasing $\alpha$ and $\rho$ .
\begin{table}[h]
\centering
 \caption{Labels for different separation and distortion criteria}

  \begin{tabular}{| l | c | }
    \hline
    Measure of good separation/low distortion & Label  \\ \hline
    $\mathds{P}( X_n(\mathcal{F}_{B_n(r)}) = 0)$ & A \\ 
    $\mathds{P}(Y_{n} \in Z_{0,n})$ (Gaussian displ.)  & B  \\ 
    $\mathds{P}(Y_{n, \delta} \in Z_{0,n})$ (Displ. at dist. $\delta$) & C \\ 
    $\mathds{E}[V_n(Z)]$ & D \\
    $\mathds{E}[V_n(Z_{0,n})]$ &  E \\ 
    $\mathds{P}(R_M(Z_{0,n}) > r )$ & G\\
    \hline
  \end{tabular}
  \label{t:summary}
\end{table}

\begin{table}[h!]
  \begin{center}
    \caption{Limit of separation and distortion metrics as $n \to \infty$ for different values of $\alpha$ and $\rho$ when $\gamma_n \sim \rho n^{\alpha}$.}
    \label{tab:table1}
    \begin{tabular}{c | c | c | c | c | m{.9cm} | m{.9cm}| m{.9cm} | c | c | c   }
 
     \multirow{2}{*}{} & $\alpha = 0$ &$\alpha \in (0, \frac{1}{2})$ & $\alpha = \frac{1}{2}$ & $\alpha \in (\frac{1}{2}, 1)$ & \multicolumn{3}{c|}{$\alpha = 1$} & $\alpha \in (1, \frac{3}{2})$ &\multicolumn{2}{c}{$\alpha = \frac{3}{2}$} \\
       & $\rho > 0$ & $\rho > 0$ & $\rho > 0$ & $\rho > 0$ & \multicolumn{3}{c|}{$\rho =  \frac{1}{\sqrt{e}}$  \hspace{.05cm}  $\rho = \frac{\pi}{\sqrt{e}}$ } & $\rho > 0$ & \multicolumn{2}{c}{$\rho_u$} \\
      \hline
      A & $e^{- \rho r}$ & 0 & 0 &0 & 0 & 0 & 0 & 0 & 0 & 0 \\
      B &$e^{- \sqrt{\frac{2}{\pi}} \rho \sigma}$ &0 & 0 &0 & 0 & 0 & 0 & 0 & 0 & 0 \\
      C & 1 &1 & $e^{- \sqrt{\frac{2}{\pi}} \rho \delta}$& 0 & 0 & 0 & 0 & 0 & 0 & 0 \\
       D & $\infty$ &$\infty$ & $\infty$ &$\infty$ & $\infty$ & 0 & 0 & 0 & 0 & 0 \\
      E & $\infty$ &$\infty$ & $\infty$ &$\infty$ & $\infty$ & $\infty$ & 0 & 0 & 0 & 0 \\
      F & 1 &1 & 1 &1 & 1 & 1 & 1 & 1 & 1 & 0
    \end{tabular}
  \end{center}
\end{table}

\begin{rem}
In the above table, the only distortion measure which was included is 
$\mathds{P}(R_M(Z_{0,n}) > r ),$ but as mentioned, we could also consider $\mathds{P}(R_M(Z_{0,n}) > \sqrt{n}r )$, which follows the information theoretic Shannon regime discussed later in Section \ref{s:commentIT}. In this case the threshold above which this probability is small in high dimensions is for $\alpha = 1$ and $\rho > \rho_u$, and by Remark \ref{r:shannon}, this is the scaling at which the centroids of the cells have intensity growing like $e^{n \lambda}$ with dimension $n$ for some $\lambda \in \R$. 
\end{rem}

\section{Dimension Reduction}\label{s:reduction}

If it is known beforehand that the data lie in a lower dimensional subspace of $\R^n$,
then the number of random hyperplanes needed to encode it may be much less than was evaluated above. If the subspace is known, we can tessellate the subspace directly. But if only the dimension of the subspace known, then we can model the subspace containing the data as a uniform random subspace in $\R^n$ independent of $X_n$. Let $\mathcal{L}$ be a random subspace in $\R^n$ of dimension $m(n)$, independent of the hyperplane tessellation $X$.
If we assume that the data all lie in $\mathcal{L}$, then instead of considering the zero cell $Z_0$ of $X$ in $\R^n$,
we can consider the zero cell $Z_0^{(\mathcal{L})}$ of the tessellation induced by the intersection of $X$ with $\mathcal{L}$.
By radial symmetry, we can just consider a fixed subspace $L$. It is known that $X \cap L$ is
a Poisson hyperplane process with intensity measure
\[\Theta_L(\cdot) = \gamma_{m} \int_{\mathds{S}_L} \int_{\R} 1\{tu + (u^{\perp} \cap L) \in \cdot \} dt \sigma_{m-1}(du),\]
where $\gamma_m = \frac{\omega_m \omega_{n+1}}{\omega_n \omega_{m+1}} \gamma$.
In \cite{HorrmanZero}, the authors showed that 
\begin{align}\label{e:VL_exp}
\mathds{E}[V_{m}(Z_0 \cap L)] = \Gamma(m + 1)\kappa_m\left(\frac{\pi \omega_n }{\gamma_n \omega_{n+1}}\right)^m,
\end{align}
and established the following results on higher moments:
\begin{align}\label{e:VL_moments}
\Gamma(m + 1)^k \kappa_m^k \left(\frac{\pi \omega_n }{\gamma_n \omega_{n+1}}\right)^{km} \leq \mathds{E}[V_m(Z_0 \cap L)^k] \leq \Gamma(2m + 1) \kappa_m^k \left(\frac{\pi \omega_n }{\gamma_n \omega_{n+1}}\right)^{2m}.
\end{align}

Proposition \ref{p:vol_thresh} can be extended to this case:
\begin{prop}
Let $\mathcal{L}_n$ be a random subspace of $\R^n$ with dimension $m_n < n$ such that $m_n \to \infty$ as $n \to \infty$. Let $X_n$ be a stationary and isotropic Poisson hyperplane process in $\R^n$ with intensity $\gamma_n$. Then, if $\gamma_n \sim \rho \sqrt{m_n n}$ for some fixed $\rho > 0$,
\[\lim_{n \rightarrow \infty} \mathds{E}[V_{m_n}(Z_{0,n} \cap \mathcal{L}_n)] = \begin{cases} 0, &\rho > \frac{\pi}{\sqrt{e}} \\ 1 , &\rho < \frac{\pi}{\sqrt{e}}. \end{cases}\]
\end{prop}


Similarly, Theorem \ref{t:verts} can be extended to:
\begin{prop}
Let $\mathcal{L}_n$ be a random subspace of $\R^n$ with dimension $m_n < n$ such that $m_n \rightarrow \infty$ as $n \rightarrow \infty$. Let $X_n$ be a stationary and isotropic Poisson hyperplane process in $\R^n$ with intensity $\gamma_n$, and let $R > 0$
Then, if $\gamma_n \sim \rho n^{\alpha - 1} m_n$ as $n \to \infty$, then there exists $\rho_{u}$ such that for all $\rho > \rho_u$,
\begin{align*}
\lim_{n \rightarrow \infty} \mathds{P}\left( R_M(Z_{0, n} \cap \mathcal{L}_n) \geq n^{\frac{3}{2} - \alpha}R \right)  = 0,
\end{align*}
and there exists $\rho_{\ell}$ such that for all $\rho < \rho_{\ell}$,
\begin{align*}
\lim_{n \rightarrow \infty} \mathds{P}\left( R_M(Z_{0, n} \cap \mathcal{L}_n) \leq n^{\frac{3}{2} - \alpha}R \right) = 0.
\end{align*}
\end{prop}


\section{Comments}\label{s:comment}
\subsection{One-Bit Compressed Sensing Comments}\label{s:commentCS}


In this paper, the compression of the data can be considered as a sequence of one-bit measurements, where each bit gives the side of a random hyperplane the data lies on. This is the paradigm of one-bit compressed sensing, and the aim of this section is to 
further connect this theory with the results in this paper.

Traditional compressed sensing is concerned with recovering a signal $x \in \R^n$ from a measurement vector $y = Ax \in \R^m$, where $A$ is some $m \times n$ measurement matrix ($m \leq n$). The goal is to find the smallest $m$ such that the signal $x$ can be recovered from $y$. If $m$ is less than $n$, this problem is ill-posed. However, Tao and Candes \cite{Tao} showed that under the assumption that $x$ is $s$-sparse, i.e. $|\mathrm{supp}(x)| \leq s$, $x$ can be recovered from $y = Ax$, where $A$ is Gaussian matrix, with $m = O\left(s \log \frac{n}{s} \right)$ measurements. 

In general the measurement vector in this set-up requires infinite bit precision. 
One-bit compressed sensing was introduced by Baraniuk and Boufounos in \cite{Bouf} and aims to recover $x$ from the most severely quantized measurements possible: $y = \mathrm{sign}(Ax)$.  This contains just one-bit per measurement. Note that taking these measurements loses all information regarding the norm of $x$, so we can only hope to recover $x/ |x|$. The goal is then to find a $x^* \in S^{n-1}$ 
such that $|x/|x| - x^*| < \delta$ for some error $\delta$.
To reconstruct the signal from $m$ measurements, Plan and Vershynin showed that one can solve the convex optimization
\begin{equation}\label{e:opt}
\min \|x\|_1 \qquad \text{ subject to}\qquad \mbox{sign}(Ax) \equiv y \text{   and   } \|Ax\|_1 = m,
\end{equation}
where $A$ is a $m \times n$ matrix with i.i.d. standard Gaussian entries, see Theorem 1.1 in \cite{PV1}. The original signal is recovered with small error if it can be guaranteed that the reconstructed signal is close in Euclidean distance to the original signal with high probability.  Plan and Vershynin showed this error guarantee specifically for sparse or almost sparse signals using the following two results. First, they showed that if the original signal is effectively sparse (see Remark 1 in \cite{PV1}), the signal returned from the optimization \eqref{e:opt} will also be effectively sparse. Second they use the fact that there is a tessellation of the signal space $S^{n-1} \cap \Sigma_s$, where $\Sigma_s := \{s-\text{sparse signals}\}$, with $m = O(s \log^2(n/s))$ hyperplanes where all cells in the tessellation will have diameter at most $\delta$, i.e., all sparse signals within a cell of the tessellation will be with $\delta$-distance apart from eachother. Thus, the recovered signal will be within distance $\delta$ of the original signal with high probability. In fact, they showed a more general result in \cite{PV2} that, for a subset $K \subseteq S^{n-1}$, all cells of a tessellation with $m \geq C\delta^{-6}\omega(K)^2$ hyperplanes will have diameter at most $\delta$ with probability as least $1 - 3e^{-c\delta^4m}$, where $\omega(K)$ is the Gaussian mean width of the set $K$.

Some recent work has shown that the same geometric techniques can be used to recover a signal $x$, both direction and magnitude, if it is known that $|x| \leq R < \infty$. Instead of linear hyperplanes tessellating $K \subset \mathds{S}^{n-1}$, consider a bounded set $K \subset \R^n$ and tessellate it with affine hyperplanes with normal vectors $a_i$ and translations from the origin $t_i$. 
It was shown  in \cite{Baraniuk} to show that a $s$-sparse signal $x$ with $|x| \leq R$ 
can be recovered with measurements of the form
\begin{align}\label{affinemeasurement}
y_i = \mathrm{sign}(\langle a_i, x \rangle - t_i), i=1,...,m,
\end{align}
where $t_1, ..., t_m \sim \mathcal{N}(0, R^2)$ are independent of $a_1, ..., a_m$.
It is proved that the following program recovers the signal with small error:
\begin{align}\label{affineprogram}
\mathrm{argmin} \|z\|_1 \qquad \text{ subject to }\qquad |z| \leq R \text{ and } y_i(\langle a_i, z \rangle - t_i) \geq 0, \qquad  \forall i = 1,..., m.
\end{align} 
More specifically, Theorem 2 in \cite{Baraniuk} states that with probability at least $1 - 3\exp(-c\delta^4 m)$, the following holds for all $x \in B_n(R) \cap \Sigma_s$: For $n \geq 2m$ and $m \geq C\delta^{-4} s\log(n/s)$, and for $y$ obtained from the measurement model \eqref{affinemeasurement}, the solution $x^*$ to the program \eqref{affineprogram} satisfies $|x - x^*| \leq \delta R$.

Also, Knudson et al. \cite{Ward} showed that if $t$ is a Gaussian vector with variance depending on $R$, $x$ can be recovered if $|x| \leq R$ by lifting to one dimension higher and using the program \eqref{e:opt}. 
They also showed you can estimate the magnitude (but not direction) of a signal $x$ in an annulus $r \leq |x| \leq R$ up to error $\delta$ with $m \gtrsim R^{4}r^{-2}\delta^{-2}$ measurements from evaluating the inverse Gaussian error function. 

If we remove the norm constraint on the signal, one can use a stationary and isotropic hyperplane tessellation to obtain an infinite sequence of one-bit measurements encoding the signal. Instead of minimizing the number of hyperplanes, the intensity of hyperplanes is minimized, as done throughout this paper for the various separation/distortion metrics. The encoding scheme corresponding to a stationary and isotropic Poisson hyperplane tessellation 
is given as follows. 
Letting $\{u_i\}_{i \in \mathds{Z}}$ be an i.i.d sequence of normal Gaussian random vectors in $\R^n$, and $\{t_i\}_{i \in \mathds{Z}}$ be the support of a Poisson point process of intensity $\gamma$ in $\R$, then the encoding is given by the one-bit measurements
\begin{equation*}
y_i = \textrm{sign}\left(\langle {u_i}/{|u_i|}, x \rangle -  t_i \right), \qquad i \in \mathds{Z}.
\end{equation*}
The collection of hyperplanes $\{H(u_i, t_i)\}_{i \in \mathds{Z}}$ tessellates all of $\R^n$ and forms a stationary and isotropic Poisson hyperplane process with intensity $\gamma$, and all data within a single cell of the tessellation have the same encoding. 
The results in the paper provide an analysis of the quality of the compression, in terms of theoretical error bounds on the separation of a typical signal from other signals or the distortion of a typical signal. These are based on some metric of the cell that a typical signal lies in, i.e., the zero cell by stationarity. 

The paradigm of one-bit compressed sensing requires the ability to recover the original data given only its one-bit encoding. Given an encoding, if one can identify a member of the cell corresponding to this sequence of bits, one can use this as an approximation of the original data. 

The convex optimization recovery technique used in the literature for the constrained norm case will return a signal $x^*$ that is one of the vertices of the cell, and knowing that all cells have small diameters ensures that recovered signal is close the original. The analogous strategy for the Poisson hyperplane compression requires showing that the vertex of the zero cell that is furthest from the origin is close in Euclidean distance, and thus the measure of distortion needed to ensure signal recovery through this convex optimization strategy is Theorem \ref{t:verts}. To ensure that the farthest vertex of the cell containing the original signal is within error distance $\delta$ the intensity of hyperplanes $\gamma_n$ must be on the order of $n^{3/2}$. 

An alternative method for reconstruction that returns a point of the cell more likely to be close to the typical signal would provide a more efficient compression. For example, if the reconstruction returns a uniformly distributed signal in the cell determined by the measurements using, for instance, the algorithm for finding an approximate uniform random point in a convex set in \cite{cousins}, this could be guaranteed to be close to the original signal with high probability using results from \cite{Oreilly}.

As seen later, a deterministic grid actually performs better than the isotropic Poisson hyperplane tessellation in the full dimensional case in the sense that a smaller constant $\rho$ is needed to ensure that the furthest vertex, or a uniform random vector in the cell, is close with high probability. However, if the data is sparse, or somehow lower-dimensional, this may make the isotropic case more desirable. In the case of a deterministic grid, only in the best case scenario will the intersection of the tessellation with a random $m$-dimensional subspace be a $m$-dimensional grid. However, in the isotropic case, the intersection will always have the distribution of a $m$-dimensional isotropic hyperplane tessellation. A more complete analysis of the case of sparse and lower dimensional data is left for future work.


\subsection{Information Theoretic Comments}\label{s:commentIT}

The aim of this section is to connect the results of the present
paper to classical information theory.

\subsubsection{Channel Coding}
Consider first channel coding. The additive noise channel features the 
transmission of codewords in $\R^n$ ($n$ is referred to the block-length of the code)
through a noisy channel. The white Gaussian noise special case is of
the same nature as that considered in Proposition \ref{pro:Gauss}: each coordinate of
a transmitted codeword is additively blurred by an independent ${\mathcal N}(0,\sigma^2)$ random variable.

In the viewpoint introduced by Poltyrev \cite{polt},
the codebook is a stationary point process in $\R^n$ (e.g., a Poisson point process in the random
coding case) and the decoding scheme consists in saying that the codeword $c$ was transmitted if the received message is in the Voronoi cell of $c$. The latter is the maximal likelihood decoder.
In the regime where the point process has intensity $e^{n\rho}$ for some $\rho \in \R$,
there is a threshold for $\rho$ below which the correct codeword is decoded with a
probability tending to one as $n$ tends to infinity, and above which the probability
of error tends to 1 as $n$ tends to infinity. In Shannon's channel coding theory, the
codewords are constrained to satisfy some power constraint requiring that the Euclidean
norm of a codeword be less than or equal to $\sqrt{nP}$, for some $P$ which is the power per symbol. 
As shown in \cite{venkat} (Lemma 2 and Theorem 7), the Poltyrev viewpoint can be connected to Shannon's channel coding theorem in the high signal to noise ratio case, namely when $P$ tends to infinity. In particular the Shannon capacity then grows like $\frac 1 2 \log(2\pi e P)$ when $P\to \infty$, and the Poltyrev capacity is what one gets asymptotically when subtracting $\frac 1 2 \log(2\pi e P)$ from the Shannon capacity.

\subsubsection{Loss-less One-bit Compression Source Coding}
Consider now source coding, which is more directly related to the setting considered
in the present paper. Consider a source with i.i.d. $N(0,\sigma^2)$ symbols. If there are $n$ such
symbols, with $n$ (also called block-length) large, they lie in a ball of radius $\sqrt{n\sigma^2}$,
which has volume about $e^{n \frac{1}{2}\log(2\pi e \sigma^2)}$. If one wants to represent in a loss-less way
all typical sequences of this type by $2^{\beta n}$ binary compression sequences,
namely all binary sequences of length $\beta n$, the volume per sequence should tend to 0. That is
$$ e^{n \frac{1}{2}\log(2\pi e \sigma^2)} e^{-\beta n \log(2)}$$ should go to 0 when $n$ tends to infinity.
This shows that the best (smallest) compression rate $\beta$ for such a signal is
$\beta_c=\frac 1 2\log(2 \pi e  \sigma^2)/\log(2)$. This is sharp and generalizes to all sources with a well
defined entropy rate. This is formalized in the source coding theorem.

In our case, we have no structure in the signal, which corresponds to letting $\sigma^2$ tend to $\infty$.
The unconstrained setting developed in the present paper can hence be seen as an {\em analogue of the Poltyrev
regime} for source coding. In addition, we focus on a {\em specific coding scheme}
which is that of Poisson hyperplanes one-bit compression.

Before going down this path, let us discuss some questions related to
coding in this one-bit compressive setting. (1) What is the codebook? A first natural answer consists in
associating one codeword sampled at random to each cell, with the uniform sampling taking place in
a conditionally independent way given the hyperplane tessellation. Another possibility is the center
of the smallest ball containing the zero cell (the out-ball). A third one is the center of the largest ball contained
in the zero cell (the in-ball). (2) What is the decoding algorithm? By this, we mean the way to retrieve the codeword,
as defined above, from the sequence of bits characterizing the cell as described in Section \ref{s:commentCS}.

For unconstrained one-bit data compression, the analogue of the Shannon threshold $\beta_c$ 
is the density $\gamma_n=\rho n^{\alpha}$ of hyperplanes that separates the situations where the 
mean volume of the typical cell tends to 0 and infinity, respectively.
As shown above, this critical density lies in the Shannon regime, namely for $\alpha=1$. More
precisely, if $\gamma_n=\rho n$, with $\rho< \rho_c=\frac 1 {\sqrt{e}}$, then this mean volume tends to infinity,
whereas if $\rho> \rho_c$, then it tends to 0. In other words, for one-bit compressive sensing based
on Poisson isotropic hyperplanes, the Palm-Shannon-Poltyrev source coding rate is $\alpha_c=1$ and $\rho_c=\frac 1 {\sqrt{e}}$.
The proposed name comes from the fact that one looks at the typical cell, with typicality
defined in the Palm sense (e.g., with respect to the point process of centers of the out-balls).
The threshold that separates the situations where the 
mean volume of zero cell tends to 0 and infinity, respectively, could be called the Feller-Shannon-Poltyrev
threshold and is obtained for a density of hyperplanes with $\alpha_c=1$ and $\rho_c=\frac {\pi}{\sqrt{e}}$.
The proposed name comes from ``Feller's paradox" which states that the interval of a stationary point process
on $\R$ containing the origin is larger than the typical interval.
The Feller-Shannon-Poltyrev rate is of the same order
as the Palm-Shannon-Poltyrev one, but $\pi$ times larger.

\subsubsection{Lossy One-bit Compression Source Coding}
In the classical lossy source coding case, one looks for a codebook 
such that the distortion between a signal and its encoding be less than or equal to $D$.
The most common distortion constraint is that the signal be at Euclidean distance order
less than or equal to $\sqrt{n D}$ from the sequence it is encoded by. The rate-distortion function
then specifies what is the best coding rate ensuring this constraint.

The framework discussed in the present paper can be seen as some Poltyrev version of lossy source coding with codebooks corresponding to one-bit data compression. 
As for the loss-less case, the first dichotomy is whether one takes the Palm viewpoint of the typical codeword 
or the Feller viewpoint of the typical data point. The cell of the former is $Z$,
whereas that containing the latter is $Z_0$.
Let us first discuss the equivalent of the classical distortion defined above in the Palm case.
If the codewords are the centers of the out-balls, then a natural definition of Palm distortion is in terms
of the radius of the out-ball of the typical cell. For instance, in this case, the rate-distortion function 
would give the smallest intensity of hyperplanes $\gamma_n=\rho n^\alpha$ such that this radius is less 
than or equal to $\sqrt{n} R$, as a function of $R$. This Palm-Shannon-Poltyrev out-ball rate-distortion
function is not known to the best of our knowledge. 
However, the Feller version of this problem is precisely solved by Theorems \ref{t:main_poi} and \ref{t:verts}.
For instance, in the case of Theorem \ref{t:verts}, the parameters in question are $\alpha=1$ and
$\rho_u(R)= x_u\frac{\sqrt{\pi}}{R\sqrt{2}},$
with $x_u$ the constant defined in the proof of the theorem. Hence the function $R\to n\rho_u(R)$
can be seen as the {\em {rate-distortion function}} for this version of the problem.
Note that for this definition of distortion,
lossy coding with a radius $R$ large enough requires a smaller hyperplane intensity
than that guaranteeing the Palm volume to go to zero (which can be seen as an analogue of loss-less coding):
the exponent is the same, namely $\alpha=1$, but the multiplicative constant $\rho(u)$ goes
to 0 as $R$ tends to infinity. As expected, relaxing the distortion constraint allows one to use smaller codes.

The paper also determines various other rate-separation functions of the Feller type.
A first instance is the Feller-Shannon-Poltyrev
in-ball function, which gives the smallest hyperplane intensity such that the closest data point not encoded
in the same way as the origin lies at a distance at least $\delta$. This last condition is equivalent to having the
radius of the largest ball centered at the origin and contained in the zero cell being larger than or equal to
$\delta$. By the same arguments as in Proposition \ref{prop:3.1}, the associated threshold is 
$\alpha_c=0$. If $\gamma_n = \rho$, the probability that this distance is at least $\delta$ is $\exp(-2\rho \delta)$.
A second example is the Feller-Shannon-Poltyrev linear contact function, which gives the smallest hyperplane
intensity such that the closest data point in some random direction and not encoded as the origin is at distance more
than $\sqrt{n D}$. By the arguments of Proposition \ref{prop:3.2}, the threshold is again $\alpha_c=0$ and 
if $\gamma_n = \rho$, the probability that this distance is at least $\sqrt{n D}$ is $\exp(-\frac{\sqrt{2}}{\sqrt{\pi}}\rho D)$.

\subsection{Why Isotropic Poisson Hyperplanes}
\label{sec:IPH}
We discuss here some mathematical reasons justifying the framework proposed here for
a one-bit compression based on Poisson isotropic hyperplanes.
Other natural options in the Poisson hyperplane framework are Poisson Manhattan hyperplanes,
where all hyperplanes are orthogonal to the orthonormal basis of $\mathds R^n$.
An even simpler hyperplane system is the square one (referred to as the deterministic grid below).
The following tables summarize the results available on
basic quantities related to these tessellations, when the distance to the nearest hyperplane is the same in expectation.
The results are proved at the end of the section.

 \def\arraystretch{2}
 \begin{table}[h]
\centering
 \caption{Comparison of quantities for different tessellations with intensity $\gamma$ in $\R^n$.}
  \begin{tabular}{ l | c | c | c    }
    Type of tessellation & $\mathds{E}[V(Z_0)]$ & $\mathds{E}[V(Z)]$ & $\mathds{P}(x \notin Z_0)$ \\ \hline
    Deterministic Grid &  $\left(\frac{2n}{\gamma}\right)^n$ & $\left(\frac{2n}{\gamma}\right)^n$ & $1_{\{\|x\|_{\infty} \geq \frac{n}{\gamma}\}}$  \\ 
   Poisson Manhattan  &$ \left(\frac{2n}{\gamma}\right)^n$ & $\frac{1}{\kappa_n}\left(\frac{n \kappa_n}{\gamma \kappa_{n-1}}\right)^n$ & $ 1-\exp(- \frac{\gamma}{n} \|x\|_1)$  \\ 
    Poisson Isotropic  & $n! \kappa_n\left(\frac{n \kappa_n}{2\gamma\kappa_{n-1}}\right)^n$ & $\frac{1}{\kappa_n}\left(\frac{n \kappa_n}{\gamma \kappa_{n-1}}\right)^n$ & $ 1-\exp\left(- \frac{2\gamma\kappa_{n-1}}{n \kappa_n} |x|\right)$  \\ 
  \end{tabular}
  \label{t:summary}
\end{table}

For all criteria in Table \ref{t:summary}, the Poisson isotropic setting outperforms the two other options.
For the expected volume of the zero cell (first column),
the isotropic Poisson tessellation is the best, i.e., has the smallest expected volume.
This fact is the main justification of the use of this Poisson isotropic structure in the context of one-bit compression: this allows the code with the smallest volume of data encoded as the typical data, among all three options.
The Poisson isotropic setting is also better than the other two in terms of the probability of separation of the typical data from data point $x$. We see from the last column that isotropic Poisson hyperplanes outperforms the other two options
orderwise: the thresholds for the latter have order $\alpha=1$, whereas that of the former
has order $\alpha=1/2$ only.

In contrast, consider now a uniform random vector $Y$ chosen in the zero cell
and take as a distortion criterion the ``norm" of $Y$, defined as
$\mathds{E}[|Y|^2]^{\frac{1}{2}}$.
The deterministic grid has the smallest norm and the Poisson grid has the second smallest norm. From Proposition 4.1 in \cite{Oreilly}, the isotropic Poisson tessellation gives an upper bound of this norm, where the upper bound is larger than the other two cases. For the quantity $R_M$, or equivalently, the furthest vertex of the zero cell from the origin, the results are the same, with the deterministic grid performing better than the Poisson grid, and the isotropic Poisson tessellation having an upper bound greater than the other two cases, since $x_{u} \approx 3$. 
For both quantities to be small, the scaling with dimension $n$ needed for $\gamma$ is $n^{3/2}$ for all three tessellations.

 \def\arraystretch{2}
 \begin{table}[h]
\centering
 \caption{Comparison of quantities for different tessellations with intensity $\gamma$ in $\R^n$.}
  \begin{tabular}{ l | c | c   }
    Type of tessellation & $\mathds{E}[|Y|^2]^{\frac{1}{2}}$ & $R_M$  \\ \hline
    Deterministic Grid & $\frac{n^{3/2}}{\sqrt{3}\gamma}$  & $\frac{n^{3/2}}{\gamma}$   \\ 
   Poisson Manhattan & $  \frac{n^{3/2}}{\gamma}$  & $\frac{\sqrt{7}n^{\frac{3}{2}}}{\sqrt{2}\gamma}$   \\ 
    Poisson Isotropic &$\lesssim \frac{\sqrt{\pi}n^{3/2}}{\sqrt{2}\gamma}$  & $\lesssim x_{u}\frac{\sqrt{\pi}n^{3/2}}{\sqrt{2} \gamma}$ \\ 
  \end{tabular}
  \label{t:summary2}
\end{table}

We now give the proofs.

To compute the norm of the uniform random vector in the zero cell of the deterministic grid, consider the fixed cube of width $\frac{2n}{\gamma}$. 
Let $Y_n \sim \mathrm{Uniform}\left(\left[-{n}/{\gamma}, {n}/{\gamma} \right]^n \right)$. Then, by the strong law of large numbers,
\begin{align*}
\frac{|Y_n|^2}{n} = \frac{\sum_{k=1}^n Y_{n,k}^2}{n} \to \mathds{E}[Y_{n,1}^2],
\end{align*}
as $n \to \infty$. Then, since $Y_{n,1} \sim \mathrm{Uniform}([-n/\gamma, n/\gamma])$,
\[\mathds{E}[Y_{n,1}^2] = \frac{1}{3}\left(\frac{n^2}{\gamma^2}- \frac{n^2}{\gamma^2} + \frac{n^2}{\gamma^2}\right) = \frac{n^2}{3\gamma^2}.\] 
Thus, $|Y_n| \sim \frac{n^{3/2}}{\sqrt{3}\gamma}$, as $n \to \infty$. The other quantities are immediate.

The Poisson Manhattan tessellation is defined as follows. Let $X$ be a Poisson hyperplane tessellation in $\R^n$ with intensity $\gamma$ and directional distribution $\phi$ that has mass $\frac{1}{2n}$ on each positive and negative axis, i.e. the normal vectors of the hyperplanes are the usual basis directions $\pm e_1, ..., \pm e_n$. Since equal weight is placed on each direction, the normal vectors of the hyperplanes form independent Poisson point processes of intensity $\frac{\gamma}{n}$ on each axis. 

For each $i = 1, \ldots n$, let $N_i = \{T_k^i\}$ be the Poisson point process of intersection points on the $\pm e_i$ axis with the usual convention that $T_0^i \leq 0 < T_1^i$. Then, the zero cell $Z_0$ of $X$ is defined as
\begin{align*}
Z_0 = \prod_{i=1}^n [T_0^i, T_1^i].
\end{align*}
Note that the interval $[T_0^i, T_1^i]$ will not have an exponential distribution, since we are requiring that 0 is in the interval, biasing for larger intervals. We obtain the distribution of the length of the interval by using the Palm distributions of $\{N_i\}_{i=1}^n$. By Slivnyak's theorem, $\mathds{P}_{N} = \mathds{P}^0_{N - \delta_0}$, so the distribution of length of the interval is the same as 
\begin{align*}
\mathds{P}(T_1^i - T_0^i \in A) = \mathds{P}^0(T_1^i + |T_{-1}^i| \in A).
\end{align*}
Under $\mathds{P}^0$, i.e. conditioned on $T_0 = 0$, $T_1$ and $|T_{-1}|$ are independent exponential random variables with parameter $\frac{\gamma}{n}$. Then, we first see that
\begin{equation*}
\E(V_n(Z_0)) = \prod_{i=1}^n \E(T_1^i - T_0^i) = \E^0(T_1^1 + |T_{-1}^1|)^n = \left(\frac{2\gamma}{n}\right)^n.
\end{equation*}
Also, for $Y$ such that conditioned on $X$, $Y \sim \mathrm{Uniform}(Z_0)$, the law of large numbers implies that as $n \rightarrow \infty$,
\begin{align*}
\frac{|Y|^2}{n} = \frac{\sum_{i=1}^n Y_i^2}{n} \rightarrow \mathds{E}[Y_1^2] \text{  a.s.}
\end{align*}
Using the fact that $(Y_i | T_0^i, T_1^i) \sim \mathrm{Uniform}([T_0^i, T_1^i])$, we have
\begin{align*}
\mathds{E}[Y_i^2] &= \mathds{E}[ \mathds{E}[Y_i^2 | T_0^i, T_1^i]] = \mathds{E}\left[ \frac{T_0^2 +  T_0T_1 + T_1^2}{3}\right] = \frac{1}{3}\left(\mathds{E}T_0^2 - \mathds{E}^0[|T_{-1}|]\mathds{E}^0[T_1] + \mathds{E}T_1^2\right) \\
&= \frac{1}{3}\left(\frac{2n^2}{\gamma^2} - \frac{n^2}{\gamma^2} + \frac{2n^2}{\gamma^2}\right) = \frac{n^2}{\gamma^2}.
\end{align*}
Thus, $|Y_n|^2 \sim \frac{n^{3/2}}{\gamma}$ as $n \to \infty$. 


For the Poisson Manhattan, the quantity $R_M$ is given by
\begin{align*}
R_M^2 = |(\max\{T_1^1, |T_0^1|\}, ..., \max \{T_1^n, |T_0^n|\}|^2 = \sum_{i=1}^n (\max\{T_1^n, |T_0^n|\}^2).
\end{align*}
By the law of large numbers, as $n \rightarrow \infty$,
\begin{align*}
\frac{R_M^2}{n} \rightarrow \mathds{E}[\max\{T_1^n, |T_0^n|\}^2], \text{  a.s.} 
\end{align*}
The distribution of $\max\{T_1, T_0\}$ is
\begin{align*}
\mathds{P}(\max\{T_1, T_0\} \leq x) = \mathds{P}^0(\max\{T_1, |T_{-1}|\} \leq x) = (1 - e^{-\frac{\gamma}{n}x})^2.
\end{align*}
Then, using integration by parts,
\begin{align*}
\mathds{E}[\max\{T_1^n, |T_0^n|\}^2] &= \int_0^{\infty} 2x \mathds{P}(\max\{T_1, |T_0|\} \geq x) dx = \int_0^{\infty} 2x(1 - (1-e^{-\frac{\gamma}{n}x})^2)dx \\
&= \int_0^{\infty} 2x(1 - (1- 2e^{-\frac{\gamma}{n}x} + e^{-\frac{2\gamma}{n}})) dx
 = \int_0^{\infty} 2x(2e^{-\frac{\gamma}{n}x} - e^{-\frac{2\gamma}{n}x}) dx \\
&=\frac{4n^2}{\gamma^2} - \frac{n^2}{2\gamma^2} = \frac{7n^2}{2\gamma^2}.
\end{align*}
Thus, $R_M$ is concentrated near $\frac{\sqrt{7}n^{3/2}}{\sqrt{2}\gamma}$ for large $n$.

\nocite{*}
\bibliographystyle{plain}
\bibliography{RefHyperplanes}

\begin{thebibliography}{10}

\bibitem{Voronoi}
Kasra Alishahi and Mohsen Sharifitabar.
\newblock Volume degeneracy of the typical cell and the chord length
  distribution for {P}oisson-{V}oronoi tessellations in high dimensions.
\newblock {\em Advances in Applied Probability}, 40:919--938, 2008.

\bibitem{venkat}
Venkat Anantharam and Fran{\c c}ois Baccelli.
\newblock Capacity and error exponents of stationary point processes under
  random additive displacements.
\newblock {\em Advances in Applied Probability}, 47(1):1--26, 2015.

\bibitem{Baraniuk}
Richard Baraniuk, Simon Foucart, Deanna Needell, Yaniv Plan, and Mary Wootters.
\newblock Exponential decay of reconstruction error from binary measurements of
  sparse signals.
\newblock {\em IEEE Transactions on Information Theory}, 63(6):3368--3385,
  March 2017.

\bibitem{Bilyk}
Dmitriy Bilyk and Michael~T. Lacey.
\newblock Random tessellations, restricted isometric embeddings, and one bit
  sensing.
\newblock {\em arXiv:1512.06697}, December 2015.

\bibitem{Bouf}
Petros~T. Boufounos and Richard~G. Baraniuk.
\newblock 1-bit compressive sensing.
\newblock Princeton, NJ, March 2008. Conference on Information Science and
  Systems (CISS).

\bibitem{calka}
Pierre Calka.
\newblock The distribution of the smallest disks containing the
  {P}oisson-{V}oronoi typical cell and the {C}rofton cell in the plane.
\newblock {\em Advances in Applied Probability}, 34:702--717, 2002.

\bibitem{calka_typical}
Pierre Calka.
\newblock Precise formulae for the distributions of the principal geometric
  characteristics of the typical cells of a two-dimensional {P}oisson-{V}oronoi
  tessellation and a {P}oisson line process.
\newblock {\em Advances in Applied Probability}, 35:551--562, 2003.

\bibitem{Tao}
Emmanuel Candes and Terence Tao.
\newblock Near-optimal recovery from random projections: Universal encoding
  strategies?
\newblock {\em IEEE Transactions on Information Theory}, 12:5406--5425, 2006.

\bibitem{Stoyan}
Sung~Nok Chiu, Dietrich Stoyan, Wilfred~S. Kendall, and Joseph Mecke.
\newblock {\em Stochastic Geometry and its Applications}.
\newblock Wiley, third edition, 2013.

\bibitem{cousins}
Ben Cousins and Santosh Vempala.
\newblock Gaussian cooling and $o^*(n^3)$ algorithms for volume and gaussian
  volume.
\newblock {\em SIAM Journal on Computing}, 47(3):1237--1273, 2018.

\bibitem{HorrmanZero}
Julia Horrman and Daniel Hug.
\newblock On the volume of the zero cell of a class of isotropic {P}oisson
  hyperplane tessellations.
\newblock {\em Advances in Applied Probability}, 46(3):622--642, 2014.

\bibitem{HugZero}
Daniel Hug, Matthias Reitzner, and Rolf Scheider.
\newblock The limit shape of the zero cell in a stationary {P}oisson hyperplane
  tessellation.
\newblock {\em The Annals of Probability}, 32(1B):1140--1167, 2004.

\bibitem{cones}
Zakhar Kabluchko, Alexander Marynych, Daniel Temesvari, and Christophe
  Th{\"{a}}le.
\newblock Cones generated by random points on half-spheres and convex hulls of
  {P}oisson point processes.
\newblock {\em arXiv:1801.08008v1}, January 2018.

\bibitem{thale18}
Zakhar Kabluchko, Christophe Th{\"{a}}le, and Dmitry Zaporozhets.
\newblock Beta polytopes {P}oisson polyhedra: f-vectors and angles.
\newblock {\em arXiv:1805.01338v1}, May 2018.

\bibitem{Ward}
Karin Knudson, Rayan Saab, and Rachel Ward.
\newblock One-bit compressive sensing with norm estimation.
\newblock {\em IEEE Transaction on Information Theory}, 62(5), May 2016.

\bibitem{Muller}
Claus M{\"u}ller.
\newblock {\em Analysis of Spherical Symmetries in {E}uclidean Spaces}, volume
  129 of {\em Applied Mathematical Sciences}.
\newblock Springer, New York, 1998.

\bibitem{Oreilly}
Eliza O'Reilly.
\newblock Thin-shell concentration for zero cells of stationary {P}oisson
  mosaics.
\newblock {\em arXiv:1809.04134}, September 2018.

\bibitem{PV1}
Yaniv Plan and Roman Vershynin.
\newblock One-bit compressed sensing by linear programming.
\newblock {\em Communications on Pure and Applied Mathematics}, 66:1275--1297,
  2013.

\bibitem{PV2}
Yaniv Plan and Roman Vershynin.
\newblock Dimension reduction by random hyperplane tessellations.
\newblock {\em Discrete and Computational Geometry}, 51:438--461, 2014.

\bibitem{polt}
Gregory Poltyrev.
\newblock On coding without restrictions for the awgn channel.
\newblock {\em IEEE Transactions on Information Theory}, 40(2):409--417, March
  1994.

\bibitem{weil}
Rolf Schneider and Wolfgang Weil.
\newblock {\em Stochastic and Integral Geometry}.
\newblock Springer, 2008.

\end{thebibliography}

\end{document}